\documentclass[12pt,reqno]{amsart}
\usepackage{amssymb}
\usepackage{amsxtra}
\usepackage{mathrsfs}
\usepackage[all]{xy}
\usepackage{cite}
\usepackage{paralist}
\usepackage[hypertex]{hyperref}
\newtheorem{theorem}{Theorem}[section]
\newtheorem{lemma}[theorem]{Lemma}
\newtheorem{prop}[theorem]{Proposition}
\newtheorem{corollary}[theorem]{Corollary}
\theoremstyle{definition}
\newtheorem{definition}[theorem]{Definition}
\theoremstyle{remark}
\newtheorem{example}[theorem]{Example}
\newtheorem{examples}[theorem]{Examples}
\newtheorem{remark}[theorem]{Remark}
\DeclareMathOperator{\Hom}{Hom}
\DeclareMathOperator{\Ker}{Ker}
\renewcommand*{\Im}{\mathop{\mathrm{Im}}}
\DeclareMathOperator{\Hol}{Hol}
\DeclareMathOperator{\Fin}{Fin}
\DeclareMathOperator{\End}{End}
\DeclareMathOperator{\Aut}{Aut}
\newcommand*{\Smash}{\mathop{\#}}
\newcommand*{\Psmash}{\mathop{\widehat{\#}}}
\newcommand*{\psmash}[1]{\mathop{\widehat{\#}^{#1}}}
\newcommand*{\Ptens}{\mathop{\widehat\otimes}}
\newcommand*{\Tens}{\mathop{\otimes}}
\newcommand*{\Free}{\mathop{*}}
\newcommand*{\Pfree}{\mathop{\widehat *}}
\newcommand*{\id}{1}
\newcommand*{\can}{\mathrm{can}}
\newcommand*{\reg}{\mathrm{reg}}
\newcommand*{\hol}{\mathrm{hol}}
\newcommand*{\alg}{\mathrm{alg}}
\newcommand*{\free}{\mathrm{free}}
\newcommand*{\add}{\mathrm{add}}
\newcommand*{\mult}{\mathrm{mult}}
\newcommand*{\wh}{\widehat}

\newcommand*{\la}{\langle}
\newcommand*{\ra}{\rangle}
\newcommand*{\CC}{\mathbb C}
\newcommand*{\N}{\mathbb N}
\newcommand*{\Z}{\mathbb Z}
\newcommand*{\R}{\mathbb R}
\newcommand*{\DD}{\mathbb D}

\newcommand*{\cO}{\mathscr O}

\newcommand*{\cF}{\mathscr F}

\newcommand*{\cB}{\mathscr B}
\newcommand*{\cI}{\mathscr I}
\newcommand*{\cL}{\mathscr L}
\newcommand*{\cU}{\mathscr U}
\newcommand*{\cV}{\mathscr V}
\newcommand*{\cT}{\mathscr T}
\newcommand*{\fg}{\mathfrak g}
\newcommand*{\bq}{\mathbf q}
\newcommand*{\bc}{\mathbf c}
\newcommand*{\bnc}{\mathbf{nc}}
\newcommand*{\spn}{\mathrm{span}}
\newcommand*{\AM}{\mathsf{AM}}
\newcommand*{\CommAM}{\mathsf{CommAM}}

\newcommand*{\Topalg}{\mathsf{Topalg}}
\newcommand*{\Ptensalg}{\Ptens\!\mbox{-}\mathsf{alg}}
\newcommand*{\eps}{\varepsilon}
\newcommand*{\ol}{\overline}
\newenvironment{mycompactenum}{\pltopsep=5pt\begin{compactenum}[\upshape (i)]}%
{\end{compactenum}}
\newcommand*{\xra}{\xrightarrow}
\newcommand{\lriso}{\stackrel{\textstyle\sim}{\smash\longrightarrow%
\vphantom{\scriptscriptstyle{_1}}}}
\begin{document}
\title[Holomorphically finitely generated algebras]{Holomorphically finitely generated algebras}
\subjclass[2010]{58B34, 32A38, 46H30, 14A22, 16S38, 16S10, 16S40}
\author{A. Yu. Pirkovskii}
\address{Faculty of Mathematics\\
National Research University Higher School of Economics\\
7 Vavilova, 117312 Moscow, Russia}
\email{aupirkovskii@hse.ru, pirkosha@gmail.com}
\thanks{This work was supported by the RFBR grant no. 12-01-00577.}
\date{}

\begin{abstract}
We introduce and study {\em holomorphically finitely generated} (HFG) Fr\'echet algebras,
which are analytic counterparts of affine (i.e., finitely generated)
$\CC$-algebras. Using a theorem of O.~Forster, we prove that
the category of commutative HFG algebras is anti-equivalent to
the category of Stein spaces of finite embedding dimension.
We also show that the class of HFG algebras is stable under some natural
constructions. This enables us to give a series of concrete examples
of HFG algebras, including Arens-Michael envelopes of affine algebras
(such as the algebras of
holomorphic functions on the quantum affine space and on the quantum torus),
the algebras of holomorphic functions
on the free polydisk, on the quantum polydisk, and on the quantum polyannulus.
\end{abstract}

\maketitle

\section{Introduction}
\label{sect:intro}

The present paper is motivated by two classical results that underlie
noncommutative geometry. The first result, which is essentially a categorical consequence
of Hilbert's Nullstellensatz, states that the category of affine algebraic varieties over $\CC$
is anti-equivalent to the category of finitely generated commutative unital $\CC$-algebras without
nilpotents. Explicitly, the anti-equivalence is given by sending
an affine variety $X$ to the algebra $\cO^\reg(X)$ of regular functions on $X$.
The second result of the same nature is the Gelfand-Naimark theorem, which
establishes an anti-equivalence between the category of locally compact Hausdorff topological
spaces and the category of commutative $C^*$-algebras by sending each locally compact space $X$
to the algebra $C_0(X)$ of continuous functions on $X$ vanishing at infinity.
Both results are traditionally viewed as starting points for noncommutative
geometry and are often mentioned in introductory textbooks on the subject
(see, e.g., \cite{Connes_NCG,Khal_bas_NCG,Manin_topics,Street,Varilly,Gracia}).

Our first goal here is to prove a complex analytic analog of the above results.
In fact, a significant step in this direction has already been taken by O.~Forster~\cite{For}.
In order to formulate his result, recall (see, e.g., \cite{GR_II})
that for each complex space $(X,\cO_X)$ the algebra $\cO(X)$ has a canonical
topology making it into a Fr\'echet algebra. If $X$ is reduced
(e.g., if $X$ is a complex manifold), then the elements of $\cO(X)$ are
holomorphic functions on $X$, and the canonical topology on $\cO(X)$
is the topology of compact convergence. By definition, a Fr\'echet algebra $A$
is a {\em Stein algebra} if $A$ is topologically isomorphic to $\cO(X)$
for some Stein space $(X,\cO_X)$.

\begin{theorem}[O.~Forster \cite{For}]
\label{thm:For}
The functor $(X,\cO_X)\mapsto\cO(X)$ is an anti-equivalence
between the category of Stein spaces and the category of Stein algebras.
\end{theorem}

Forster's theorem, although somewhat similar to the Nullstellensatz and to the
Gelfand-Naimark Theorem, still lacks an important feature of the above results.
Indeed, it does not identify the category of Stein algebras with the ``commutative part''
of any category of algebras, and so it
does not give us any idea of what ``noncommutative Stein spaces'' could be.
Our goal is to propose a possible approach to this problem for the
special case of Stein spaces having finite embedding dimension.

In Section~\ref{sect:HFG}, we introduce a category of ``holomorphically finitely generated''
(HFG for short) Fr\'echet algebras. Such algebras may be viewed as ``analytic counterparts''
of affine (i.e., finitely generated) $\CC$-algebras. We show that a commutative
Fr\'echet algebra is holomorphically finitely generated if and only if it is topologically
isomorphic to the algebra $\cO(X)$ for a Stein space $(X,\cO_X)$ of finite embedding
dimension. Together with Forster's Theorem, this implies that the category of
Stein spaces of finite embedding dimension in anti-equivalent to the category of
commutative HFG algebras. We hope that such a ``refinement'' of Forster's Theorem
may contribute to the development of noncommutative complex analytic geometry, a field
much less investigated than other types of noncommutative geometry
(cf. discussion in \cite{Beggs_Smith,Khal_Landi_van_Su}).

In Sections~\ref{sect:free}--\ref{sect:skew}, we show that the category of HFG algebras
is stable under a number of natural constructions.
Section~\ref{sect:free} is devoted to Arens-Michael free products \cite{Cuntz_doc}.
We give several explicit constructions
for the Arens-Michael free product, and we show that the Arens-Michael free
product of finitely many HFG algebras is again an HFG algebra.
In Section~\ref{sect:smash}, we consider two analytic versions of the smash product
construction, the $\Ptens$-smash product \cite{Pir_locsolv} and the Arens-Michael smash product.
The former is more explicit and easier to deal with, while the latter is better suited for
the category of Arens-Michael algebras.
We give sufficient conditions for the two smash products to coincide, and we
show that the Arens-Michael smash product of HFG algebras is an HFG algebra.
In view of further applications, we pay a special attention to the case
of semigroup actions and semigroup graded Arens-Michael algebras.
Algebras of skew holomorphic functions, or Arens-Michael Ore extensions \cite{Pir_qfree},
are discussed in Section~\ref{sect:skew}.
They can be viewed as ``analytic counterparts'' of algebraic Ore extensions,
and are often reduced to
Arens-Michael smash products. Given a balanced domain
(respectively, a Reinhardt domain) $D\subset\CC^n$ and an HFG
algebra $A$ endowed with an equicontinuous action $\sigma$ of $\Z_+$ (respectively, $\Z^n$),
we define the algebra $\cO(D,A;\sigma)$ of skew $A$-valued holomorphic functions on $D$,
and we show that $\cO(D,A;\sigma)$ is an HFG algebra.

Section~\ref{sect:examples} is entirely devoted to examples of HFG algebras.
One source of examples comes from the notion of an Arens-Michael envelope \cite{T1,X2}.
In Subsection~\ref{subsect:AM}, we observe that
the Arens-Michael envelope of any finitely generated
algebra is holomorphically finitely generated. As a consequence, the algebras $\cO_\bq(\CC^n)$
and $\cO_\bq((\CC^\times)^n)$ of
holomorphic functions on the quantum affine space and on the quantum torus \cite{Pir_qfree,Pir_wdg},
the algebra of entire functions on a basis of a Lie algebra \cite{Dos_hdim,Dos_absbas},
and some other algebras studied in \cite{Pir_qfree} are holomorphically finitely generated.
Other examples are specializations of the constructions discussed
in Sections~\ref{sect:free}--\ref{sect:skew}.
In Subsection~\ref{subsect:free_poly}, we define the HFG algebra $\cF(\DD^n_R)$
of holomorphic functions on the free polydisk, and we give an explicit
power series representation of it. Deformed algebras of holomorphic functions
on products of balanced and Reinhardt domains are discussed in Subsection~\ref{subsect:q-prod}.
Such algebras, being special cases of Arens-Michael Ore extensions, are also HFG algebras.
Subsections~\ref{subsect:q_polydisk} and~\ref{subsect:q_polyann} are devoted to the
algebras $\cO_\bq(\DD^n_R)$ and $\cO_\bq(\DD^n_{r,R})$ of holomorphic functions
on the quantum polydisk and on the quantum polyannulus,
respectively. The former is shown to be a quotient of $\cF(\DD^n_R)$, while the latter
is an iterated Arens-Michael Ore extension of the algebra of holomorphic functions
on the one-dimensional annulus. As a consequence, both $\cO_\bq(\DD^n_R)$
and $\cO_\bq(\DD^n_{r,R})$ are HFG algebras.

Some of the results of the present paper were announced in \cite{Pir_Shulman}.

\section{Preliminaries and notation}
\label{sect:prelim}

We shall work over the field $\CC$ of complex numbers. All algebras are assumed to
be associative and unital, and all algebra homomorphisms are assumed to be unital
(i.e., to preserve identity elements). In particular, a subalgebra of an algebra $A$ is assumed
to contain the identity of $A$.

By a {\em topological algebra} we mean an algebra $A$ endowed with a topology
making $A$ into a topological vector space and such that the multiplication $A\times A\to A$
is separately continuous.
A {\em Fr\'echet algebra} is a complete metrizable locally convex topological
algebra (i.e., a topological algebra whose underlying space
is a Fr\'echet space). A {\em locally $m$-convex algebra} \cite{Michael} is a topological
algebra $A$ whose topology can be defined by a family of submultiplicative
seminorms (i.e., seminorms $\|\cdot\|$ satisfying $\| ab\|\le \| a\| \| b\|$
for all $a,b\in A$). Equivalently, a topological algebra $A$ is locally $m$-convex if and only if
there exists a base $\cU$ of $0$-neighborhoods in $A$ such that each $U\in\cU$ is absolutely
convex and idempotent (i.e., satisfies $U^2\subseteq U$).
A complete locally $m$-convex algebra is called
an {\em Arens-Michael algebra} \cite{X1}.
The semigroup of all continuous endomorphisms of a topological algebra $A$ will be denoted by
$\End(A)$, and the group of all topological automorphisms of $A$ will be denoted by $\Aut(A)$.

The completed projective tensor product of locally convex spaces
(see, e.g., \cite{Groth,Kothe_II,X1}) will be denoted by $\Ptens$.
A complete locally convex topological algebra $A$ with jointly continuous multiplication
$A\times A\to A$ is called a {\em $\Ptens$-algebra} \cite{X1}.
For example, each Fr\'echet algebra and each Arens-Michael algebra are $\Ptens$-algebras.
The multiplication on a $\Ptens$-algebra $A$ uniquely extends to a continuous
linear map $\mu_A\colon A\Ptens A\to A$ satisfying the usual associativity conditions.
In other words, a $\Ptens$-algebra is the same thing as an algebra in the tensor category
of complete locally convex spaces endowed with the bifunctor $\Ptens$.
The notions of {\em $\Ptens$-coalgebra, $\Ptens$-bialgebra, Hopf $\Ptens$-algebra,
left/right $\Ptens$-module} over a $\Ptens$-algebra etc. have obvious meanings:
roughly, we just replace $\Tens$ by $\Ptens$ in the usual algebraic definitions.
In fact, such notions can be defined in a straightforward manner in any symmetric tensor
category (cf. discussion in \cite{Majid_brd}).
By an {\em Arens-Michael bialgebra} we mean a $\Ptens$-bialgebra that is locally $m$-convex
as a topological algebra. Similarly, if $H$ is a $\Ptens$-bialgebra, then by a {\em left $H$-$\Ptens$-module
Arens-Michael algebra} we mean a left $H$-$\Ptens$-module algebra that is locally $m$-convex
as a topological algebra. If $A$ is a $\Ptens$-algebra, we let
$\eta_A\colon\CC\to A$ denote the unit map taking $\lambda\in\CC$ to $\lambda 1\in A$.
If $C$ is a $\Ptens$-coalgebra, then the comultiplication (respectively, the counit)
on $C$ will be denoted by $\Delta_C$ (respectively, $\eps_C$). Finally, if $H$ is a Hopf
$\Ptens$-algebra, then the antipode of $H$ will be denoted by $S_H$.

Let $A$ be a topological algebra. Recall from \cite{X2} (cf. \cite{T1}) that the {\em Arens-Michael envelope}
of $A$ is the completion of $A$ with respect to the family of all continuous submultiplicative
seminorms on $A$. Thus we have a canonical continuous homomorphism
$i_A\colon A\to\wh{A}$, and it is easy to show \cite{T1,X2}
that for each Arens-Michael algebra $B$
and each continuous homomorphism $\varphi\colon A\to B$
there exists a unique continuous homomorphism $\wh{\varphi}\colon\wh{A}\to B$
such that $\wh{\varphi}\circ i_A=\varphi$. In other words, we have a natural bijection
\begin{equation}
\label{AM_def}
\Hom_{\AM}(\wh{A},B)\cong\Hom_{\Topalg}(A,B)\qquad (B\in\AM),
\end{equation}
where $\Topalg$ is the category of topological algebras and $\AM\subset\Topalg$ is the full
subcategory consisting of Arens-Michael algebras.
Moreover, the correspondence $A\mapsto\wh{A}$ is a functor from $\Topalg$
to $\AM$, and this functor is left adjoint to the inclusion $\AM\hookrightarrow\Topalg$ (see~\eqref{AM_def}).
In what follows, the functor $A\mapsto\wh{A}$ will be called the {\em Arens-Michael functor}.

Here is an important special case.
Given an algebra $A$, we can always make $A$ into a topological algebra
$A_{\mathrm{str}}$ by endowing $A$ with the strongest locally convex topology
(i.e., the topology generated by all seminorms on $A$). We define the Arens-Michael
envelope of $A$ to be the Arens-Michael envelope of $A_{\mathrm{str}}$.

Arens-Michael envelopes were introduced by J.~L.~Taylor \cite{T1} under the name of
``completed locally $m$-convex envelopes''. Now it is customary to call them
``Arens-Michael envelopes'', following the terminology suggested by A.~Ya.~Helemskii \cite{X2}.
For a more detailed study of Arens-Michael envelopes we refer to
\cite{Pir_qfree,Pir_stbflat,Pir_locsolv,Dos_holfun,Dos_Slodk,Dos_hdim,Dos_locleft,Dos_absbas}.

Throughout we will use the following multi-index notation. Let $\Z_+=\N\cup\{ 0\}$ denote the set
of all nonnegative integers. For each $n\in\N$ and each $d\in \Z_+$, let $W_{n,d}=\{ 1,\ldots ,n\}^d$,
and let $W_n=\bigsqcup_{d\in\Z_+} W_{n,d}$. Thus a typical element of $W_n$
is a $d$-tuple $\alpha=(\alpha_1,\ldots ,\alpha_d)$ of arbitrary length
$d\in\Z_+$, where $\alpha_j\in\{ 1,\ldots ,n\}$ for all $j$. The only element
of $W_{n,0}$ will be denoted by $*$.
For each $\alpha\in W_{n,d}\subset W_n$, let $|\alpha|=d$.
Given an algebra $A$, an $n$-tuple $a=(a_1,\ldots ,a_n)\in A^n$, and
$\alpha=(\alpha_1,\ldots ,\alpha_d)\in W_n$, we let
$a_\alpha=a_{\alpha_1}\cdots a_{\alpha_d}\in A$ if $d>0$;
it is also convenient to set $a_*=1\in A$.
Given $k=(k_1,\ldots ,k_n)\in\Z_+^n$, we let $a^k=a_1^{k_1}\cdots a_n^{k_n}$.
If the $a_i$'s are invertible, then $a^k$ makes sense for all $k\in\Z^n$.
As usual, for each $k=(k_1,\ldots ,k_n)\in\Z^n$ we let $|k|=|k_1|+\cdots +|k_n|$.

Let $S$ be a set, and let $\Fin(S)$ denote the collection of all finite subsets of $S$
ordered by inclusion. Suppose that $E$ is a complete locally convex space,
and that $(x_s)_{s\in S}$ is a family
of elements of $E$. For each $I\in\Fin(S)$, let $x_I=\sum_{s\in I} x_s$.
By definition~\cite[1.3.8]{Pietsch_nucl}, the family $(x_s)_{s\in S}$ is {\em summable}
if the net $(x_I)_{I\in\Fin(S)}$ converges in $E$. The limit of the above net is denoted by
$\sum_{s\in S} x_s$. The family $(x_s)_{s\in S}$ is {\em absolutely summable}
\cite[1.4.1]{Pietsch_nucl} if for each continuous seminorm $\|\cdot\|$ on $E$
we have $\sum_{s\in S} \| x_s\|<\infty$. Each absolutely summable
family in $E$ is summable \cite[1.4.5]{Pietsch_nucl}.

Let $\cO(\CC^n)$ denote the Fr\'echet algebra of all holomorphic functions on $\CC^n$.
Recall that for each Arens-Michael algebra $A$ and each commuting $n$-tuple
$a=(a_1,\ldots ,a_n)\in A^n$ there exists an {\em entire functional calculus},
i.e., a unique continuous homomorphism $\gamma_a^\hol\colon\cO(\CC^n)\to A$
taking the complex coordinates $z_1,\ldots ,z_n$ to $a_1,\ldots ,a_n$, respectively.
Explicitly, $\gamma_a^\hol$ is given by
\[
\gamma_a^\hol(f)=f(a)=\sum_{k\in\Z_+^n} c_k a^k
\qquad\Bigl( f=\sum_{k\in\Z_+^n} c_k z^k\in\cO(\CC^n)\Bigr).
\]
In the language of category theory, the above result means that the algebra $\cO(\CC^n)$
together with the $n$-tuple $(z_1,\ldots ,z_n)\in\cO(\CC^n)^n$ represents the functor
$A\mapsto A^n$ acting from the category $\CommAM$ of commutative
Arens-Michael algebras to the category
of sets. Thus we have a natural bijection
\begin{equation}
\label{univ_O}
\Hom_\CommAM(\cO(\CC^n),A)\cong A^n \qquad (A\in\CommAM).
\end{equation}
Equivalently, we have $\CC[z_1,\ldots ,z_n]\sphat\;\cong\cO(\CC^n)$ (cf. \cite{T2}).

A noncommutative (or, more exactly, free) version of the entire functional
calculus was introduced by J.~L.~Taylor \cite{T2}.
Let $F_n=\CC\la \zeta_1,\ldots ,\zeta_n\ra$ be the free algebra with generators
$\zeta_1,\ldots ,\zeta_n$. The set
$\{\zeta_\alpha : \alpha\in W_n\}$ of all words in $\zeta_1,\ldots ,\zeta_n$
is the standard vector space basis of $F_n$.
The algebra of {\em free entire functions} \cite{T2,T3} is defined by
\begin{equation}
\label{Fn}
\cF_n=\Bigl\{ a=\sum_{\alpha\in W_n} c_\alpha\zeta_\alpha :
\| a\|_\rho=\sum_{\alpha\in W_n} |c_\alpha|\rho^{|\alpha|}<\infty
\;\forall\rho>0\Bigr\}.
\end{equation}
The topology on $\cF_n$ is given by the norms $\|\cdot\|_\rho\; (\rho>0)$,
and the multiplication is given by concatenation (like on $F_n$).
Each norm $\|\cdot\|_\rho$ is easily seen to be submultiplicative, so
$\cF_n$ is a Fr\'echet-Arens-Michael algebra containing $F_n$ as a dense
subalgebra. As was observed by D.~Luminet \cite{Lum}, $\cF_n$ is nuclear.
Note also that $\cF_1$ is topologically isomorphic to $\cO(\CC)$.

Taylor \cite{T2} observed that for each Arens-Michael algebra $A$
and each $n$-tuple
$a=(a_1,\ldots ,a_n)\in A^n$ there exists a {\em free entire functional calculus},
i.e., a unique continuous homomorphism $\gamma_a^\free\colon\cF_n\to A$
taking the free generators $\zeta_1,\ldots ,\zeta_n$ to $a_1,\ldots ,a_n$, respectively.
Explicitly, $\gamma_a^\free$ is given by
\begin{equation}
\label{free_entire}
\gamma_a^\free(f)=f(a)=\sum_{\alpha\in W_n} c_\alpha a_\alpha
\qquad \Bigl( f=\sum_{\alpha\in W_n} c_\alpha \zeta_\alpha\in\cF_n\Bigr).
\end{equation}
Similarly to the commutative case, the above result means that the algebra $\cF_n$
together with the $n$-tuple $(\zeta_1,\ldots ,\zeta_n)\in\cF_n^n$ represents the functor
$A\mapsto A^n$ acting from the category $\AM$ of
Arens-Michael algebras to the category
of sets. Thus we have a natural bijection
\begin{equation}
\label{univ_F}
\Hom_\AM(\cF_n,A)\cong A^n \qquad (A\in\AM).
\end{equation}
Equivalently, we have $\wh{F}_n\cong\cF_n$ (cf. \cite{T2}).

Given $n\in\N$ and $R=(R_1,\ldots ,R_n)\in (0,+\infty]^n$, let
\[
\DD^n_R=\{ z=(z_1,\ldots ,z_n)\in\CC^n : |z_i|<R_i\;\forall i=1,\ldots ,n\}
\]
denote the open polydisk in $\CC^n$ of polyradius $R$.
If $R_i=\rho\in (0,+\infty]$ for all $i$, then we write $\DD^n_\rho$ for $\DD^n_R$.
In particular, $\DD^n_\infty=\CC^n$.
If $r,R\in [0,+\infty]^n$ and $r_i<R_i$ for all $i$, let
\[
\DD^n_{r,R}=\{ z=(z_1,\ldots ,z_n)\in\CC^n : r_i<|z_i|<R_i\;\forall i=1,\ldots ,n\}
\]
denote the open polyannulus of polyradii $r$ and $R$.
Clearly, $\DD^n_R=\DD_{R_1}^1\times\cdots\times\DD_{R_n}^1$
and $\DD^n_{r,R}=\DD_{r_1,R_1}^1\times\cdots\times\DD_{r_n,R_n}^1$.

The structure sheaf of a complex space $X$ will always be denoted by $\cO_X$,
while the structure sheaf of a scheme $S$ will be denoted by $\cO_S^\reg$.
If $X$ is a complex manifold and $E$ is a complete locally convex space, then
$\cO(X,E)$ stands for the space of all holomorphic $E$-valued functions on $X$.
Recall (see, e.g., \cite{Groth}) that there exists a topological isomorphism
$\cO(X)\Ptens E\cong\cO(X,E)$ taking each $f\otimes u\in\cO(X)\Ptens E$
to the function $x\mapsto f(x)u$. Recall also (loc. cit.) that, if $X$ and $Y$ are complex manifolds,
then we have a topological isomorphism
$\cO(X)\Ptens\cO(Y)\cong\cO(X\times Y)$ taking each
$f\otimes g\in\cO(X)\Ptens \cO(Y)$ to the function $(x,y)\mapsto f(x)g(y)$.

\section{HFG algebras and a refinement of Forster's theorem}
\label{sect:HFG}

We start this section with some elementary observations on the algebra $\cF_n$.
Letting $A=\cO(\CC^n)$ in \eqref{univ_F}, we obtain a continuous homomorphism
\begin{equation}
\label{pi_n}
\pi_n\colon\cF_n\to\cO(\CC^n),\quad \zeta_i\mapsto z_i\quad (i=1,\ldots ,n).
\end{equation}
The following is immediate from the universal properties of $\cO(\CC^n)$ and $\cF_n$.

\begin{prop}
\label{prop:calc_compat}
For each commutative Arens-Michael algebra $A$ and each $a\in A^n$ we have a commutative
diagram
\[
\xymatrix@R=10pt{
\cF_n \ar[dr]^{\gamma_a^\free} \ar[dd]_{\pi_n} \\
& A\\
\cO(\CC^n) \ar[ur]_{\gamma_a^\hol}
}
\]
\end{prop}

Given a topological algebra $A$, let $\ol{[A,A]}$ denote the closure of the
two-sided ideal of $A$ generated by the commutators $[a,b]=ab-ba\; (a,b\in A)$.

\begin{prop}
\label{prop:F_quot_O}
The map $\pi_n$ given by \eqref{pi_n} induces a topological algebra isomorphism
\[
\cF_n/\ol{[\cF_n,\cF_n]}\lriso\cO(\CC^n),\quad
\bar\zeta_i=\zeta_i+\ol{[\cF_n,\cF_n]}\mapsto z_i\quad (i=1,\ldots ,n).
\]
\end{prop}
\begin{proof}
Each continuous homomorphism from $\cF_n$ to a commutative Arens-Michael algebra
factors through $\cF_n/\ol{[\cF_n,\cF_n]}$. Combined with~\eqref{univ_F}, this implies that
$\cF_n/\ol{[\cF_n,\cF_n]}$ together with the $n$-tuple
$(\bar\zeta_1,\ldots ,\bar\zeta_n)$ has the same universal property~\eqref{univ_O}
as $\cO(\CC^n)$. The rest is clear.
\end{proof}

Let $A$ and $B$ be Arens-Michael algebras, and let $\varphi\colon A\to B$ be a continuous
homomorphism. Given $a=(a_1,\ldots ,a_n)\in A^n$, we denote the $n$-tuple
$(\varphi(a_1),\ldots ,\varphi(a_n))\in B^n$ simply by $\varphi(a)$.

\begin{prop}
\label{prop:hom_calc}
For each $a\in A^n$ and each $f\in\cF_n$, we have $\varphi(f(a))=f(\varphi(a))$.
\end{prop}
\begin{proof}
This is immediate either from the power series expansion~\eqref{free_entire}
or from the observation that
$\varphi\circ\gamma_a^\free\colon\cF_n\to B,\; f\mapsto\varphi(f(a))$, is a continuous homomorphism
taking each $\zeta_i$ to $\varphi(a_i)$.
\end{proof}

Let us introduce some notation.
Given an Arens-Michael algebra $A$, an $m$-tuple $a\in A^m$, and an $n$-tuple
$f=(f_1,\ldots ,f_n)\in\cF_m^n$, let
$f(a)=(f_1(a),\ldots ,f_n(a))\in A^n$. If now $g\in\cF_n^k$, we define
$g\circ f\in\cF_m^k$ by $g\circ f=g(f)$. Thus we obtain a
``free superposition'' operation
\begin{equation}
\label{superpos}
\cF_n^k\times\cF_m^n\to\cF_m^k,\quad (g,f)\mapsto g\circ f.
\end{equation}

\begin{prop}
\label{prop:assoc}
For each Arens-Michael algebra $A$, each $a\in A^m$, each $f\in\cF_m^n$,
and each $g\in\cF_n^k$ we have $(g\circ f)(a)=g(f(a))$.
\end{prop}
\begin{proof}
We may assume that $k=1$. Consider the maps $\varphi,\psi\colon\cF_n\to A$
given by
\[
\varphi(g)=g(f(a)),\quad \psi(g)=(g\circ f)(a)\quad (g\in\cF_n).
\]
Clearly, $\varphi=\gamma_{f(a)}^\free$ is a continuous homomorphism, and so is
$\psi=\gamma_a^\free\circ\gamma_f^\free$. For each $i=1,\ldots ,n$ we have
\begin{align*}
\varphi(\zeta_i)&=\zeta_i(f(a))=f_i(a),\\
\psi(\zeta_i)&=(\zeta_i\circ f)(a)=(\zeta_i(f))(a)=f_i(a).
\end{align*}
Therefore $\varphi=\psi$, as required.
\end{proof}

\begin{corollary}
The free superposition \eqref{superpos} is associative.
\end{corollary}
\begin{proof}
Take $f\in\cF_m^n$, $g\in\cF_n^k$, and $h\in\cF_k^\ell$. Using Proposition~\ref{prop:assoc},
we obtain
\[
(h\circ g)\circ f=(h\circ g)(f)=h(g(f))=h\circ (g\circ f).\qedhere
\]
\end{proof}

From now on, we assume that $A$ is an Arens-Michael algebra.

\begin{definition}
We say that a subalgebra $B\subseteq A$ is {\em holomorphically closed}
if for each $n\in\N$, each $b\in B^n$, and each $f\in\cF_n$ we have
$f(b)\in B$.
\end{definition}

In the commutative case, we can use $\cO(\CC^n)$ instead of $\cF_n$:

\begin{prop}
Let $A$ be a commutative Arens-Michael algebra.
Then a subalgebra $B\subseteq A$ is holomorphically closed if and only
if for each $n\in\N$, each $b\in B^n$, and each $f\in\cO(\CC^n)$ we have
$f(b)\in B$.
\end{prop}
\begin{proof}
By Propositions \ref{prop:calc_compat} and \ref{prop:F_quot_O},
we have
\begin{equation}
\label{twosets}
\{ f(b) : f\in\cO(\CC^n)\}=\{ g(b) : g\in\cF_n\}.
\end{equation}
The rest is clear.
\end{proof}

\begin{examples}
Of course, if $B$ is closed in $A$, then it is holomorphically closed.
More generally, if $B$ is an Arens-Michael algebra under a topology
stronger than the topology inherited from $A$, then $B$ is holomorphically
closed in $A$. For example, the algebra $C^\infty(M)$ of smooth functions
on a manifold $M$ is holomorphically closed in $C(M)$.
\end{examples}

\begin{remark}
Our notion of a holomorphically closed subalgebra should not be confused with
the more common notion of a subalgebra stable under the holomorphic functional
calculus (see, e.g., \cite{Connes_NCG,Gracia}). Recall that a subalgebra
$B\subseteq A$ is stable under the holomorphic functional
calculus if for each $b\in B$, each neighborhood $U$ of the spectrum $\sigma_A(b)$,
and each $f\in\cO(U)$ we have $f(b)\in B$. Such a subalgebra is
necessarily spectrally invariant in $A$, i.e., $\sigma_B(b)=\sigma_A(b)$ for all
$b\in B$. In contrast, a holomorphically closed subalgebra need not be spectrally
invariant. For example, for each domain $D\subseteq\CC$ the algebra
$\cO(\CC)$ may be viewed (via the restriction map)
as a holomorphically closed subalgebra of $\cO(D)$,
but it is not spectrally invariant in $\cO(D)$ unless $D=\CC$.
\end{remark}

It is clear from the definition that the intersection of any family of holomorphically
closed subalgebras of $A$ is holomorphically closed. This leads naturally to the
following definition.

\begin{definition}
The {\em holomorphic closure}, $\Hol(S)$, of a subset $S\subseteq A$ is
the intersection of all holomorphically closed subalgebras of $A$ containing $S$.
\end{definition}

Clearly, $\Hol(S)$ is the smallest holomorphically closed subalgebra of $A$ containing $S$.
It can also be described more explicitly as follows.

\begin{prop}
\label{prop:hol_closure_expl}
For each nonempty subset $S\subseteq A$ we have
\begin{equation}
\label{hol_closure_expl}
\Hol(S)=\{ f(a) : f\in\cF_n,\; a\in S^n,\; n\in\Z_+\}.
\end{equation}
\end{prop}
\begin{proof}
Let $B$ denote the right-hand side of \eqref{hol_closure_expl}.
Clearly, $S\subseteq B\subseteq\Hol(S)$, so it suffices to show that $B$
is a holomorphically closed subalgebra of $A$.
First observe that $1_A\in B$ (take any $a\in S$ and $f=1_{\cF_1}$).
Let now $a\in S^m$, $b\in S^n$, $f\in\cF_m$, and $g\in \cF_n$.
We have to show that $f(a)+g(b)\in B$ and $f(a)g(b)\in B$.
To this end, define $F,G\in\cF_{m+n}$ by
\[
F=f(\zeta_1,\ldots ,\zeta_m),\quad G=g(\zeta_{m+1},\ldots ,\zeta_{m+n}),
\]
where the $\zeta_i$'s are now considered as elements of $\cF_{m+n}$.
We clearly have $f(a)=F(a,b)$ and $g(b)=G(a,b)$. Therefore
\[
f(a)+g(b)=F(a,b)+G(a,b)=(F+G)(a,b)\in B,
\]
and similarly $f(a)g(b)\in B$. Thus $B$ is a subalgebra of $A$.

Let us now show that $B$ is holomorphically closed. Take any $b\in B^k$, and represent it
in the form
\[
b=\bigl(f_1(a^{(1)}),\ldots ,f_k(a^{(k)})\bigr),
\]
where $a^{(i)}\in S^{n_i}$ and $f_i\in\cF_{n_i}$ $(i=1,\ldots ,k)$.
Let $n=n_1+\cdots +n_k$.
By replacing $f_1,\ldots ,f_k$ with suitable elements $\bar f_1,\ldots ,\bar f_k$ of $\cF_n$
(where $\bar f_1=f_1(\zeta_1,\ldots ,\zeta_{n_1})$,
$\bar f_2=f_2(\zeta_{n_1+1},\ldots ,\zeta_{n_1+n_2})$ etc.),
we see that $b=\bar f(a)$, where
\[
\bar f=(\bar f_1,\ldots ,\bar f_k)\in\cF_n^k,\quad
a=(a^{(1)},\ldots ,a^{(k)})\in S^n.
\]
Now Proposition~\ref{prop:assoc} implies that
for each $g\in\cF_k$ we have $g(b)=g(\bar f(a))=(g\circ\bar f)(a)\in B$.
Thus $B$ is holomorphically closed, as required.
\end{proof}

\begin{corollary}
If $A$ is commutative, then for each nonempty subset $S\subseteq A$ we have
\[
\Hol(S)=\{ f(a) : f\in\cO(\CC^n),\; a\in S^n,\; n\in\Z_+\}.
\]
\end{corollary}
\begin{proof}
This is immediate from \eqref{hol_closure_expl}
and~\eqref{twosets}.
\end{proof}

\begin{corollary}
\label{cor:Hol_hom}
Let $A$ and $B$ be Arens-Michael algebras, and let $\varphi\colon A\to B$ be a continuous
homomorphism. Then for each $S\subseteq A$ we have $\varphi(\Hol(S))=\Hol(\varphi(S))$.
As a consequence, if $A_0\subseteq A$ is a holomorphically closed subalgebra, then
$\varphi(A_0)$ is holomorphically closed in $B$.
\end{corollary}
\begin{proof}
Apply $\varphi$ to both sides of \eqref{hol_closure_expl} and use Proposition~\ref{prop:hom_calc}.
\end{proof}

In the case where $S$ is finite, an even more explicit characterization of $\Hol(S)$
is possible.

\begin{prop}
\label{prop:hol_closure_fin_expl}
For each $a=(a_1,\ldots ,a_n)\in A^n$ we have
\begin{equation}
\label{hol_closure_fin_expl}
\Hol(\{ a_1,\ldots ,a_n\})=\{ f(a) : f\in\cF_n\}.
\end{equation}
\end{prop}
\begin{proof}
Let $B$ denote the right-hand side of \eqref{hol_closure_fin_expl}, and let $S=\{ a_1,\ldots ,a_n\}$.
Clearly, $S\subseteq B\subseteq\Hol(S)$,
and $B$ is a subalgebra of $A$. So it suffices to show that $B$
is holomorphically closed in $A$.
Each element $b\in B^k$ has the form $b=f(a)$ for some $f\in\cF_n^k$.
Thus for each $g\in \cF_k$ we have
$g(b)=g(f(a))=(g\circ f)(a)\in B$.
Thus $B$ is holomorphically closed, as required.
\end{proof}

\begin{corollary}
\label{cor:hol_closure_fin_expl_comm}
If $A$ is commutative, then for each $a=(a_1,\ldots ,a_n)\in A^n$ we have
\[
\Hol(\{ a_1,\ldots ,a_n\})=\{ f(a) : f\in\cO(\CC^n)\}.
\]
\end{corollary}

Now we are ready to introduce our main objects of study.

\begin{definition}
Let $A$ be a Fr\'echet-Arens-Michael algebra.
We say that $A$ is {\em holomorphically generated} by a subset $S\subseteq A$ if
$\Hol(S)=A$. We say that $A$ is {\em holomorphically finitely generated} (HFG for short)
if $A$ is holomorphically generated by a finite subset.
\end{definition}

\begin{example}
\label{ex:HFG_free}
By Proposition \ref{prop:hol_closure_fin_expl}, $\cF_n$ is holomorphically generated
by $\zeta_1,\ldots ,\zeta_n$. Similarly, Corollary~\ref{cor:hol_closure_fin_expl_comm}
shows that $\cO(\CC^n)$ is holomorphically generated by $z_1,\ldots ,z_n$.
Thus $\cF_n$ and $\cO(\CC^n)$ are HFG algebras.
\end{example}

\begin{prop}
\label{prop:HFG_quot}
If $A$ is an HFG algebra and $I\subseteq A$ is a closed two-sided ideal, then
$A/I$ is also an HFG algebra.
\end{prop}
\begin{proof}
Clear from Corollary~\ref{cor:Hol_hom}.
\end{proof}

\begin{remark}
A closed subalgebra of an HFG algebra is not necessarily an HFG algebra
(see Remark~\ref{rem:non-HFG_closed} below).
\end{remark}

HFG algebras can also be characterized as follows.

\begin{prop}
\label{prop:HFG_quot_F}
Let $A$ be a Fr\'echet-Arens-Michael algebra.
\begin{mycompactenum}
\item $A$ is holomorphically finitely generated if and only if
$A$ is topologically isomorphic to $\cF_n/I$ for some $n\in\Z_+$ and for some closed
two-sided ideal $I\subseteq \cF_n$.
\item If $A$ is commutative, then $A$ is holomorphically finitely generated if and only if
$A$ is topologically isomorphic to $\cO(\CC^n)/I$ for some $n\in\Z_+$ and for some closed
two-sided ideal $I\subseteq \cO(\CC^n)$.
\end{mycompactenum}
\end{prop}
\begin{proof}
(i) The ``if'' part is immediate from Example~\ref{ex:HFG_free} and~Proposition~\ref{prop:HFG_quot}.
Conversely, if $A$ is an HFG algebra holomorphically generated by $a_1,\ldots ,a_n\in A$,
then Proposition~\ref{prop:hol_closure_fin_expl} implies that
$\gamma_a^\free\colon\cF_n\to A$ (where $a=(a_1,\ldots ,a_n)$) is onto.
Applying the Open Mapping Theorem, we obtain a topological algebra isomorphism
$A\cong\cF_n/\Ker\gamma_a^\free$. This completes the proof of (i).
The proof of (ii) is similar (use Corollary~\ref{cor:hol_closure_fin_expl_comm}
instead of~Proposition \ref{prop:hol_closure_fin_expl}).
\end{proof}

\begin{corollary}
Each HFG algebra is nuclear.
\end{corollary}

The following theorem characterizes commutative HFG algebras as algebras of holomorphic functions.

\begin{theorem}
\label{thm:comm_HFG}
A commutative Fr\'echet-Arens-Michael algebra is holomorphically finitely generated if and only if
it is topologically isomorphic to $\cO(X)$ for some Stein space $(X,\cO_X)$ of finite
embedding dimension.
\end{theorem}
\begin{proof}
By the Remmert-Bishop-Narasimhan-Wiegmann
Embedding Theorem~\cite{Wiegmann}, each Stein space $(X,\cO_X)$ of finite
embedding dimension is biholomorphically equivalent to a closed analytic subspace of $\CC^n$,
where $n\in\N$ is large enough. By H.~Cartan's Theorem~B, the restriction map
$\cO(\CC^n)\to\cO(X)$ is onto. Applying Proposition~\ref{prop:HFG_quot_F}~(ii)
and using the Open Mapping Theorem, we conclude that $\cO(X)$ is an HFG algebra.

Conversely, let $A$ be a commutative HFG algebra.
Taking into account Proposition~\ref{prop:HFG_quot_F}~(ii), we may assume that
$A=\cO(\CC^n)/I$, where $I\subseteq\cO(\CC^n)$ is a closed two-sided ideal.
By \cite[Satz 2]{For_prim}, $\cO(\CC^n)/I\cong\cO(X)$ for a certain closed
analytic subspace $X$ of $\CC^n$ (specifically, $X$ is the zero set of $I$,
and $\cO_X=\cO_{\CC^n}/\cI$, where $\cI\subseteq\cO_{\CC^n}$ is the ideal sheaf
generated by $I$). Clearly, $X$ has finite embedding dimension.
\end{proof}

Combining Theorem~\ref{thm:comm_HFG} with Forster's Theorem~\ref{thm:For},
we obtain the following.

\begin{theorem}
\label{thm:comm_HFG_cat}
The functor $(X,\cO_X)\mapsto\cO(X)$ is an anti-equivalence
between the category of Stein spaces of finite embedding dimension and the category
of commutative HFG algebras.
\end{theorem}

Theorem~\ref{thm:comm_HFG_cat} looks similar to the Gelfand-Naimark Theorem
and to the categorical consequence of the Nullstellensatz (see Section~\ref{sect:intro}).
Thus HFG algebras may be considered
as candidates for ``noncommutative Stein spaces of finite embedding dimension''.
Of course, this naive point of view needs a solid justification, and it is perhaps too early to say whether
it will lead to an interesting theory. As a first step towards this goal, we will give some concrete
examples of HFG algebras in Section~\ref{sect:examples}, showing thereby
that the class of HFG algebras is rather large. To this end, we will need to show that the class
of HFG algebras is stable under a number of
natural constructions. This will be done in Sections~\ref{sect:free}-\ref{sect:skew}.

\section{Free products}
\label{sect:free}

\begin{definition}[cf. \cite{Cuntz_doc}]
Let $(A_i)_{i\in I}$ be a family of Arens-Michael algebras.
The {\em Arens-Michael free product} of $(A_i)_{i\in I}$
is the coproduct of $(A_i)_{i\in I}$ in the category of Arens-Michael algebras, i.e.,
an Arens-Michael algebra $A$ together with a natural isomorphism
\begin{equation}
\label{Pfree_def}
\Hom_{\AM}(A,B)\cong\prod_{i\in I}\Hom_{\AM}(A_i,B)\qquad (B\in\AM).
\end{equation}
More explicitly, the Arens-Michael free product of $(A_i)_{i\in I}$ is an Arens-Michael algebra
$A$ together with a family
$(j_i\colon A_i\to A)_{i\in I}$ of continuous homomorphisms
such that for each Arens-Michael algebra
$B$ and each family $(\varphi_i\colon A_i\to B)_{i\in I}$ of continuous homomorphisms
there exists a unique continuous homomorphism
$A\to B$ making the following diagram (for each $i\in I$) commute:
\[
\xymatrix@C-5pt@R-5pt{
A \ar@{-->}[rr] && B\\
& A_i \ar[ul]^{j_i} \ar[ur]_{\varphi_i}
}
\]
\end{definition}

Clearly, if the Arens-Michael free product $A$ of $(A_i)_{i\in I}$ exists, then it is unique up to a unique
topological algebra isomorphism over the $A_i$'s.
In what follows, we will write $A=\Pfree_{i\in I} A_i$.

To show that $\Pfree_{i\in I} A_i$ exists, recall the definition
(due to J.~Cuntz~\cite{Cuntz_doc}; cf. also \cite{Vogt1,Vogt2,Pir_qfree})
of an Arens-Michael tensor algebra.
Let $E$ be a complete locally convex space, and let $\{\|\cdot\|_\lambda : \lambda\in\Lambda\}$
be a directed defining family of seminorms on $E$.
For each $\lambda\in\Lambda$ and each $n\in\Z_+$, let $\|\cdot\|_\lambda^{(n)}$
denote the $n$th projective tensor power of $\|\cdot\|_\lambda$
(we let $\|\cdot\|_\lambda^{(0)}=|\cdot|$ by definition).
The {\em Arens-Michael tensor algebra} (or a {\em smooth tensor algebra})
$\wh{T}(E)$ is defined by
\[
\wh{T}(E)=\Bigl\{ a=\sum_{n=0}^\infty a_n : a_n\in E^{\wh{\otimes}n},\;
\| a\|_{\lambda,\rho}=\sum_n \| a_n\|_\lambda^{(n)} \rho^n<\infty\; \forall\lambda\in\Lambda,
\;\forall\rho>0\Bigr\}.
\]
The topology on $\wh{T}(E)$ is defined by the seminorms
$\|\cdot\|_{\lambda,\rho}\; (\lambda\in\Lambda,\;\rho>0)$, and the multiplication
on $\wh{T}(E)$ is given by concatenation, like on the usual tensor algebra $T(E)$.
Each seminorm $\|\cdot\|_{\lambda,\rho}$ is easily seen to be
submultiplicative, and so $\wh{T}(E)$ is an Arens-Michael algebra containing
$T(E)$ as a dense subalgebra. As was observed by Cuntz \cite{Cuntz_doc},
$\wh{T}(E)$ has the universal property that, for every Arens-Michael algebra $A$,
each continuous linear map $E\to A$ uniquely extends to a continuous
homomorphism $\wh{T}(E)\to A$. In other words, there is a natural isomorphism
\begin{equation}
\label{univ_T}
\Hom_\AM(\wh{T}(E),A)\cong\cL(E,A)\qquad (A\in\AM),
\end{equation}
where $\cL(E,A)$ is the space of all continuous linear maps from $E$ to $A$.
Note that $\wh{T}(\CC^n)\cong\cF_n$, and that~\eqref{univ_F}
is a special case of~\eqref{univ_T}.

\begin{prop}
Let $(A_i)_{i\in I}$ be a family of Arens-Michael algebras, and let $A$
be the completion of the quotient of $\wh{T}=\wh{T}(\bigoplus_{i\in I} A_i)$
by the two-sided closed ideal $J$ generated by all elements of the form
\[
a_i\otimes b_i-a_i b_i,\quad 1_{\wh{T}}-1_{A_i}\quad
(a_i,b_i\in A_i,\; i\in I).
\]
For each $i\in I$, define $j_i\colon A_i\to A$ to be the composition of the
canonical embedding of $A_i$ into $\wh{T}$
with the quotient map $\wh{T}\to A$. Then $j_i$ is a continuous homomorphism,
and $A$ together with $(j_i\colon A_i\to A)_{i\in I}$
is the Arens-Michael free product of $(A_i)_{i\in I}$.
\end{prop}
\begin{proof}
Clearly, each $j_i$ is a continuous linear map, and
it is immediate from the definition of $J$ that $j_i$ is an algebra homomorphism.
For each Arens-Michael algebra $B$ we have natural isomorphisms
\[
\begin{split}
\Hom_{\AM}(A,B)
&\cong\{\varphi\in\Hom_{\AM}(\wh{T},B) : \varphi|_J=0\}\\
&\cong\bigl\{\psi\in\cL\bigl(\bigoplus_{i\in I} A_i,B\bigr) :
\psi|_{A_i}\text{ is an algebra homomorphism }\forall i\in I\bigr\}\\
&\cong\prod_{i\in I}\Hom_{\AM}(A_i,B).\qedhere
\end{split}
\]
\end{proof}

A more explicit construction
of $\Pfree_{i\in I} A_i$ (in the nonunital category, and in the case where $I$ is finite)
was given by Cuntz~\cite{Cuntz_doc}.
We need to adapt his construction to the unital case, assuming that for each $i\in I$
there exists a closed two-sided ideal $A_i^\circ\subset A_i$ such that
$A_i=\CC 1_{A_i}\oplus A_i^\circ$ as locally convex spaces.

For each $d\ge 2$, let
\begin{equation}
\label{I_d}
I_d=\Bigl\{ \alpha=(\alpha_1,\ldots ,\alpha_d)\in I^d : \alpha_k\ne\alpha_{k+1}
\;\forall k=1,\ldots ,d-1\Bigr\}.
\end{equation}
Let also $I_1=I$, $I_0=\{ *\}$ (the singleton), and $I_\infty=\bigsqcup_{d\in\Z_+} I_d$.
Given $d>0$ and $\alpha=(\alpha_1,\ldots ,\alpha_d)\in I_\infty$, let
$A_\alpha=A^\circ_{\alpha_1}\Ptens\cdots\Ptens A^\circ_{\alpha_d}$.
For $d=0$, we let $A_*=\CC$. For each $i\in I$,
choose a directed defining family $\{\|\cdot\|_{\nu} : \nu\in\Lambda_i\}$
of submultiplicative seminorms on $A^\circ_i$.
Let $\Lambda=\prod_{i\in I}\Lambda_i$, and, for each
$\lambda=(\lambda_i)_{i\in I}\in\Lambda$ and each $\alpha=(\alpha_1,\ldots ,\alpha_d)\in I_\infty$,
let $\|\cdot\|_\lambda^{(\alpha)}$ denote the projective tensor product of the
seminorms $\|\cdot\|_{\lambda_{\alpha_1}},\ldots ,\|\cdot\|_{\lambda_{\alpha_k}}$.
For $d=0$, it is convenient to set $\|\cdot\|_\lambda^{(*)}=|\cdot|$.
Let $P(I)= [1,+\infty)^I$ denote the family of all functions on $I$ with values in $[1,+\infty)$.
Given $p\in P(I)$, we let $p_i=p(i)$ for $i\in I$,
$p_\alpha=p_{\alpha_1}\cdots p_{\alpha_d}$ for $\alpha=(\alpha_1,\ldots ,\alpha_d)\in I_\infty$
with $d>0$, and $p_*=1$. Let
\begin{equation}
\label{Pfree_expl}
A=
\Bigl\{ a=\sum_{\alpha\in I_\infty} a_\alpha : a_\alpha\in A_\alpha,\;
\| a\|_{\lambda,p}=\sum_{\alpha\in I_\infty} \| a_\alpha\|_\lambda^{(\alpha)} p_\alpha<\infty
\;\forall \lambda\in\Lambda,\; \forall p\in P(I)\Bigr\}.
\end{equation}
A standard argument shows that $A$ is a complete locally convex space
with respect to the topology determined by the seminorms
$\|\cdot\|_{\lambda,p}\; (\lambda\in\Lambda,\; p\in P(I))$.

Now, given $\alpha=(\alpha_1,\ldots ,\alpha_d)\in I_\infty$, let
$A^\alg_\alpha=A^\circ_{\alpha_1}\Tens\cdots\Tens A^\circ_{\alpha_d}$.
The algebraic sum
\[
A^\alg=\bigoplus_{\alpha\in I_\infty} A^\alg_\alpha
\]
is clearly a dense subspace of $A$. Recall from \cite{Nica_Sp} that there is a multiplication
on $A^\alg$ such that each $A_i$ becomes a subalgebra of $A^\alg$ and such that
$A^\alg=\Free_{i\in I} A_i$, the algebraic free product of $(A_i)_{i\in I}$.
Specifically, the multiplication on $A^\alg$ is given by concatenation composed (if necessary) with the
product maps $A^\circ_i\Tens A^\circ_i\to A^\circ_i\; (i\in I)$.

\begin{prop}
\label{prop:Pfree_expl}
The multiplication on $A^\alg$ uniquely extends to a continuous multiplication on $A$.
The resulting algebra $A$ together with the canonical embeddings $A_i\to A$
is the Arens-Michael free product of $(A_i)_{i\in I}$. Moreover, each seminorm
$\|\cdot\|_{\lambda,p}\; (\lambda\in\Lambda,\; p\in P(I))$ is submultiplicative.
\end{prop}
\begin{proof}
Let us show that each seminorm $\|\cdot\|_{\lambda,p}$
is submultiplicative on $A^\alg$.
Take any $a\in A^\circ_\alpha$ and $b\in A^\circ_\beta$, where
$\alpha=(\alpha_1,\ldots ,\alpha_d)\in I_\infty$,
$\beta=(\beta_1,\ldots ,\beta_m)\in I_\infty$.
Suppose first that $\alpha_d\ne\beta_1$. Then $\gamma=(\alpha,\beta)\in I_\infty$,
$ab=a\otimes b\in A^\circ_\alpha\Tens A^\circ_\beta=A^\circ_\gamma$, and
\begin{equation}
\label{submult_free_1}
\| ab\|_{\lambda,p}=\| a\otimes b\|_\lambda^{(\gamma)} p_\gamma
=\| a\|_\lambda^{(\alpha)} \| b\|_\lambda^{(\beta)} p_\alpha p_\beta
=\| a\|_{\lambda,p} \| b\|_{\lambda,p}.
\end{equation}
Now consider the case where $\alpha_d=\beta_1$. Let
$\gamma=(\alpha_1,\ldots ,\alpha_d,\beta_2,\ldots ,\beta_m)\in I_\infty$.
Denote by $\mu\colon A^\circ_{\alpha_d}\Tens A^\circ_{\beta_1}\to A^\circ_{\alpha_d}$
the product map on $A^\circ_{\alpha_d}$, and let
\[
\bar\mu=\id\otimes\mu\otimes\id\colon A^\circ_\alpha\Tens A^\circ_\beta\to A^\circ_\gamma.
\]
Then $ab=\bar\mu(a\otimes b)\in A^\circ_\gamma$.
Clearly, for each $u\in A^\circ_\alpha\Tens A^\circ_\beta$
we have $\| \bar\mu(u)\|_\lambda^{(\gamma)} \le \| u\|_\lambda^{(\alpha,\beta)}$,
where $\|\cdot\|_\lambda^{(\alpha,\beta)}$ is the projective tensor product of the
seminorms $\|\cdot\|_\lambda^{(\alpha)}$ and $\|\cdot\|_\lambda^{(\beta)}$.
Therefore
\begin{equation}
\label{submult_free_2}
\| ab\|_{\lambda,p}=\| \bar\mu(a\otimes b)\|_\lambda^{(\gamma)} p_\gamma
\le\| a\|_\lambda^{(\alpha)} \| b\|_\lambda^{(\beta)} p_\alpha p_\beta
=\| a\|_{\lambda,p} \| b\|_{\lambda,p}.
\end{equation}
Finally, take any $a,b\in A^\alg$ and decompose them as
\[
a=\sum_{\alpha\in I_\infty} a_\alpha,\quad
b=\sum_{\beta\in I_\infty} b_\beta \qquad (a_\alpha\in A^\circ_\alpha,\; b_\beta\in A^\circ_\beta).
\]
By using \eqref{submult_free_1} and \eqref{submult_free_2}, we obtain
\[
\| ab\|_{\lambda,p}
=\bigl\| \sum_{\alpha,\beta} a_\alpha b_\beta\bigr\|_{\lambda,p}
\le \sum_{\alpha,\beta} \| a_\alpha\|_{\lambda,p} \| b_\beta\|_{\lambda,p}
=\| a\|_{\lambda,p} \| b\|_{\lambda,p}.
\]
Thus each seminorm $\|\cdot\|_{\lambda,p}$ is submultiplicative on $A^\alg$,
and so $A^\alg$ is a locally $m$-convex algebra. Since $A$ is the completion of $A^\alg$,
it follows that the multiplication on $A^\alg$ uniquely extends to a continuous
multiplication on $A$ making $A$ into an Arens-Michael algebra.
Moreover, each seminorm $\|\cdot\|_{\lambda,p}$ is submultiplicative on $A$.

Let now $B$ be an Arens-Michael algebra, and let $(\varphi_i\colon A_i\to B)_{i\in I}$
be a family of continuous homomorphisms. Since $A^\alg=\Free_{i\in I} A_i$,
there exists a unique algebra homomorphism $\varphi^\alg\colon A^\alg\to B$
such that $\varphi^\alg|_{A_i}=\varphi_i$ for all $i\in I$. We claim that $\varphi^\alg$
is continuous with respect to the topology inherited from $A$.
Indeed, let $\|\cdot\|$ be a continuous submultiplicative seminorm on $B$.
Then for each $i\in I$ there exists $p_i\ge 1$ and $\lambda_i\in\Lambda_i$ such that
for each $a\in A_i$ we have
\begin{equation}
\label{phi_i_cont}
\|\varphi_i(a)\|\le p_i \| a\|_{\lambda_i}.
\end{equation}
Given $\alpha=(\alpha_1,\ldots ,\alpha_d)\in I_\infty$, let
$\varphi_\alpha^\alg=\varphi^\alg|_{A^\alg_\alpha}$.
By construction (cf. \cite{Nica_Sp}), $\varphi_\alpha^\alg$ is the composition of the maps
\[
A_{\alpha_1}^\circ\Tens\cdots\Tens A_{\alpha_d}^\circ
\xra{\varphi_{\alpha_1}\Tens\cdots\Tens\varphi_{\alpha_d}}
B \Tens \cdots \Tens B \xra{\mu_d} B,
\]
where $\mu_d$ is the product map. Now it follows from~\eqref{phi_i_cont}
and from the submultiplicativity of $\|\cdot\|$ that for each $a\in A^\alg_\alpha$ we have
\begin{equation*}
\| \varphi_\alpha^\alg(a) \| \le p_\alpha \| a\|_\lambda^{(\alpha)},
\end{equation*}
where $\lambda=(\lambda_i)_{i\in I}\in\Lambda$ and $p=(p_i)_{i\in I}\in P(I)$.
Finally, for each $a=\sum_\alpha a_\alpha\in A^\alg$ (where $a_\alpha\in A^\alg_\alpha$)
we obtain
\[
\bigl\| \varphi^\alg\bigl(\sum_\alpha a_\alpha\bigr) \bigr\|
\le\sum_\alpha \|\varphi_\alpha^\alg(a_\alpha) \|
\le \sum_\alpha \| a_\alpha\|_\lambda^{(\alpha)} p_\alpha
=\| a\|_{\lambda,p}.
\]
Therefore $\varphi^\alg$ is continuous, and so it uniquely extends to a continuous
homomorphism $\varphi\colon A\to B$. Thus $A=\Pfree_{i\in I} A_i$, as required.
\end{proof}

Given $\alpha\in I_d\subset I_\infty$, let $|\alpha|=d$.

\begin{corollary}
\label{cor:Pfree_expl_fin}
Under the above assumptions, suppose that $I$ is finite. Then
\begin{equation}
\label{Pfree_expl_fin}
\Pfree_{i\in I} A_i=
\Bigl\{ a=\sum_{\alpha\in I_\infty} a_\alpha : a_\alpha\in A_\alpha,\;
\| a\|_{\lambda,\tau}=\sum_{\alpha\in I_\infty} \| a_\alpha\|_\lambda^{(\alpha)} \tau^{|\alpha|}<\infty
\;\forall \lambda\in\Lambda,\; \forall \tau\ge 1\Bigr\}.
\end{equation}
Moreover, each seminorm
$\|\cdot\|_{\lambda,\tau}\; (\lambda\in\Lambda,\; \tau\ge 1)$ is submultiplicative.
\end{corollary}
\begin{proof}
For each $\tau\ge 1$ we have $\|\cdot\|_{\lambda,\tau}=\|\cdot\|_{\lambda,p}$, where
$p_i=\tau$ for all $i\in I$. Conversely,
for each $p\in P(I)$ we have $\|\cdot\|_{\lambda,p}\le \|\cdot\|_{\lambda,\tau}$, where
$\tau=\max_{i\in I} p_i$.
Now~\eqref{Pfree_expl_fin} follows from~\eqref{Pfree_expl} and
Proposition~\ref{prop:Pfree_expl}.
\end{proof}

Here is a simple example.

\begin{prop}
\label{prop:Fn-Fm}
For each family $(E_i)_{i\in I}$ of complete locally convex spaces there exists a topological
algebra isomorphism
\[
\Pfree_{i\in I} \wh{T}(E_i)\cong\wh{T}\bigl(\bigoplus_{i\in I} E_i\bigr).
\]
In particular,
\begin{align}
\label{Fn-Fm}
&\cF_m\Pfree\cF_n\cong\cF_{m+n},\\
\label{F_O_free}
&\cF_n \cong \cO(\CC)\Pfree\cdots\Pfree\cO(\CC) \quad\text{($n$ factors)}.
\end{align}
\end{prop}
\begin{proof}
This is immediate from~\eqref{Pfree_def}, from~\eqref{univ_T}, and from the fact that
the locally convex direct sum is the coproduct in the category of locally convex spaces.
\end{proof}

The next observation shows that the Arens-Michael free product commutes with quotients
(modulo completions).

\begin{prop}
\label{prop:Pfree_quot}
Let $(A_i)_{i\in I}$ be a family of Arens-Michael algebras, and let, for each $i\in I$,
$J_i$ be a closed two-sided ideal of $A_i$. Denote by $J$ the closed two-sided ideal
of $\Pfree_{i\in I} A_i$ generated by all the $J_i$'s. There exists a topological algebra
isomorphism
\[
\Pfree_{i\in I} \bigl((A_i/J_i)^\sim\bigr) \cong \bigl((\Pfree_{i\in I} A_i)/J\bigr)^\sim.
\]
\end{prop}
\begin{proof}
For each Arens-Michael algebra $B$ we have natural isomorphisms
\[
\begin{split}
\Hom_\AM\bigl(\bigl((\Pfree_{i\in I} A_i)/J\bigr)^\sim,B\bigr)
&\cong\bigl\{ \varphi\in\Hom_\AM\bigl(\Pfree_{i\in I} A_i,B\bigr) : \varphi(J)=0\bigr\}\\
&\cong\bigl\{ \varphi\in\Hom_\AM\bigl(\Pfree_{i\in I} A_i,B\bigr) : \varphi(J_i)=0\;\forall i\in I\bigr\}\\
&\cong\Bigl\{ \varphi=(\varphi_i)_{i\in I}\in\prod_{i\in I}\Hom_\AM(A_i,B) :
\varphi_i(J_i)=0\;\forall i\in I\Bigr\}\\
&\cong\prod_{i\in I}\Hom_\AM\bigl((A_i/J_i)^\sim,B\bigr). \qedhere
\end{split}
\]
\end{proof}

\begin{corollary}
\label{cor:HFG_free}
If $A_1$ and $A_2$ are HFG algebras, then so is $A_1\Pfree A_2$.
\end{corollary}
\begin{proof}
By Proposition~\ref{prop:HFG_quot_F}, each $A_k\; (k=1,2)$ is a quotient of $\cF_{n_k}$
for some $n_k\in \N$. Applying Proposition~\ref{prop:Pfree_quot} and using~\eqref{Fn-Fm},
we see that $A_1\Pfree A_2$ is a quotient of $\cF_{n_1+n_2}$.
Hence $A_1\Pfree A_2$ is an HFG algebra by Proposition~\ref{prop:HFG_quot_F}.
\end{proof}

\section{Smash products}
\label{sect:smash}

In this section we consider two analytic versions of the smash product construction
(see, e.g., \cite{Sweedler,Dasc_Nast,GH,Majid_brd}), the {\em $\Ptens$-smash product}
and the {\em Arens-Michael smash product}. While the former is a special
case of smash products in tensor categories \cite{Majid_brd}, the latter appears to be new.
Let us start by recalling some definitions \cite{Majid_boson,Majid_brd}.

Let $H$ be a $\Ptens$-bialgebra.
Similarly to the algebraic case (see, e.g., \cite{Kassel}), the projective tensor product
of two left $H$-$\Ptens$-modules (respectively, comodules) has a natural structure of
a left $H$-$\Ptens$-module (respectively, comodule).
A {\em left $H$-$\Ptens$-module algebra} is
a $\Ptens$-algebra $A$ endowed with the structure of a left $H$-$\Ptens$-module in such
a way that the product $\mu_A\colon A\Ptens A\to A$ and the unit map
$\eta_A\colon \CC\to A$ are $H$-$\Ptens$-module morphisms.
Similarly, a {\em left $H$-$\Ptens$-comodule algebra} is
a $\Ptens$-algebra $B$ endowed with the structure of a left $H$-$\Ptens$-comodule in such a way that
the product $\mu_B\colon B\Ptens B\to B$ and the unit map
$\eta_B\colon \CC\to B$ are $H$-$\Ptens$-comodule morphisms.
Given a left $H$-$\Ptens$-module algebra $A$ and a left $H$-$\Ptens$-comodule algebra $B$,
the {\em $\Ptens$-smash product} $A\psmash{H} B$
is defined as follows. As a topological vector space, $A\psmash{H} B$ is equal to $A\Ptens B$.
To define multiplication,
denote by $\mu_{H,A}\colon H\Ptens A\to A$ the action of $H$ on $A$,
by $\Delta_{H,B}\colon B\to H\Ptens B$ the coaction of $H$ on $B$,
and define $\tau_{B,A}\colon B\Ptens A\to A\Ptens B$ to be the composition
\begin{equation}
\label{tau}
B\Ptens A \xra{\Delta_{H,B}\otimes\id_A} H\Ptens B\Ptens A \xra{\id_H\otimes c_{B,A}}
H\Ptens A\Ptens B \xra{\mu_{H,A}\otimes\id_B} A\Ptens B
\end{equation}
(here $c_{B,A}$ stands for the flip $B\Ptens A\to A\Ptens B$).
Then the map
\begin{equation}
\label{smashprod}
(A\Ptens B)\Ptens (A\Ptens B) \xra{\id_A\otimes\tau_{B,A}\otimes\id_B}
A\Ptens A\Ptens B\Ptens B \xra{\mu_A\otimes\mu_B} A\Ptens B
\end{equation}
is an associative multiplication on $A\Ptens B$
(the associativity was proved in \cite{Majid_boson} for the special case $B=H$;
the proof extends {\em verbatim} to an arbitrary $B$).
The resulting $\Ptens$-algebra
is denoted by $A\psmash{H} B$ and is called the {\em $\Ptens$-smash product} of $A$ and $B$
over $H$.

To motivate our further constructions, let us show that $A\psmash{H} B$ is characterized
by a universal property. We need a simple lemma.

\begin{lemma}
\label{lemma:id_smash}
Let $H$ be a $\Ptens$-bialgebra, $A$ be a left $H$-$\Ptens$-module algebra, and $B$ be a left
$H$-$\Ptens$-comodule algebra.
We have the following identities in $A\psmash{H} B$:
\begin{align}
\label{b=1}
(a\otimes 1)(a'\otimes b')&=aa'\otimes b',\\
\label{a'=1}
(a\otimes b)(1\otimes b')&=a\otimes bb',\\
\label{tau_prod}
(1\otimes b)(a\otimes 1)&=\tau_{B,A}(b\otimes a),\\
\label{prod_tau}
(a\otimes b)(a'\otimes b')&=(a\otimes 1)\tau_{B,A}(b\otimes a')(1\otimes b').
\end{align}
\end{lemma}
\begin{proof}
Identities \eqref{b=1} and \eqref{tau_prod} are immediate from~\eqref{smashprod}.
To prove~\eqref{a'=1}, let $\eps'\colon H\Ptens\CC\to\CC$ be given by
$\eps'(h\otimes\lambda)=\eps_H(h)\lambda$. We have a commutative diagram
\begin{equation}
\label{a'=1_2}
\xymatrix@C+25pt{
B\Ptens A\ar[r]^{\Delta_{H,B}\otimes\id_A} & H\Ptens B\Ptens A \ar[r]^{\id_H\otimes c_{B,A}}
& H\Ptens A\Ptens B \ar[r]^(.6){\mu_{H,A}\otimes\id_B} & A\Ptens B\\
B\Ptens \CC \ar[r]^{\Delta_{H,B}\otimes\id_\CC} \ar[u]_{\id_B\otimes\eta_A}
& H\Ptens B\Ptens \CC \ar[r]^{\id_H\otimes c_{B,\CC}} \ar[u]_{\id_H\otimes\id_B\otimes\eta_A}
& H\Ptens \CC\Ptens B \ar[r]^{\eps'\otimes\id_B} \ar[u]_{\id_H\otimes\eta_A\otimes\id_B}
& \CC\Ptens B \ar[u]_{\eta_A\otimes\id_B}\\
B \ar[r]^{\Delta_{H,B}} \ar[u]_{\can} & H\Ptens B \ar@{=}[r] \ar[u]_{\can}
& H\Ptens B \ar[r]^{\eps_H\otimes\id_B} \ar[u]_{\can} & B \ar[u]_{\can}
}
\end{equation}
where the arrows marked by ``$\can$'' are the canonical isomorphisms.
Since $(\eps_H\otimes\id_B)\Delta_{H,B}=\id_B$, \eqref{a'=1} now follows from~\eqref{a'=1_2}.
Identity~\eqref{prod_tau} is immediate from~\eqref{b=1}--\eqref{tau_prod}.
\end{proof}

Lemma~\ref{lemma:id_smash} implies, in particular, that the maps
\begin{equation}
\label{i_A_i_B}
\begin{split}
i_A&\colon A\to A\psmash{H} B,\quad a\mapsto a\otimes 1,\\
i_B&\colon B\to A\psmash{H} B,\quad b\mapsto 1\otimes b,
\end{split}
\end{equation}
are $\Ptens$-algebra homomorphisms.

\begin{definition}
Given a $\Ptens$-algebra $C$, we say that a pair $(\varphi_A\colon A\to C,\;\varphi_B\colon B\to C)$
of $\Ptens$-algebra homomorphisms is {\em compatible} if
the diagram
\[
\xymatrix{
B\Ptens A \ar[rr]^{\tau_{B,A}} \ar[d]_{\varphi_B\otimes\varphi_A}
&& A\Ptens B \ar[d]^{\varphi_A\otimes\varphi_B} \\
C\Ptens C \ar[dr]_{\mu_C} && C\Ptens C \ar[dl]^{\mu_C}\\
& C
}
\]
is commutative.
\end{definition}

\begin{prop}
\label{prop:Psmash_univ}
Let $H$ be a $\Ptens$-bialgebra, $A$ be a left $H$-$\Ptens$-module algebra, and $B$ be a left
$H$-$\Ptens$-comodule algebra.
Let the homomorphisms $i_A$, $i_B$ be given by~\eqref{i_A_i_B}. Then
\begin{mycompactenum}
\item the pair $(i_A,i_B)$ is compatible;
\item for each $\Ptens$-algebra $C$ and each compatible pair
$(\varphi_A\colon A\to C,\;\varphi_B\colon B\to C)$
of continuous homomorphisms there exists a unique continuous
homomorphism $A\psmash{H} B\to C$ making the diagram
\begin{equation}
\label{Psmash_univ}
\xymatrix@-10pt{
& A\psmash{H} B \ar@{-->}[dd] \\
A \ar[ur]^(.4){i_A} \ar[dr]_{\varphi_A} && B \ar[ul]_(.4){i_B} \ar[dl]^{\varphi_B} \\
& C
}
\end{equation}
commute.
\end{mycompactenum}
In other words, we have a natural isomorphism
\[
\Hom_{\Ptensalg}(A\psmash{H} B,C)\cong\bigl\{ \text{\upshape compatible }
(\varphi_A,\varphi_B)\in\Hom_{\Ptensalg}(A,C)\times \Hom_{\Ptensalg}(B,C)\bigr\},
\]
where $\Ptensalg$ is the category of all $\Ptens$-algebras.
\end{prop}
\begin{proof}
(i) It is immediate from \eqref{b=1} and \eqref{tau_prod} that
\[
\mu_{A\psmash{H} B}(i_A\otimes i_B)=\id_{A\Ptens B},\quad
\mu_{A\psmash{H} B}(i_B\otimes i_A)=\tau_{B,A}.
\]
This readily implies that the pair $(i_A,i_B)$ is compatible.

(ii) Define $\psi\colon A\psmash{H} B\to C$ by $\psi=\mu_C(\varphi_A\otimes\varphi_B)$.
Clearly, $\psi$ makes \eqref{Psmash_univ} commute. We claim that $\psi$ is an algebra
homomorphism. Indeed, by using~\eqref{b=1}, for each $a,a'\in A$ and each $b'\in B$ we obtain
\[
\psi((a\otimes 1)(a'\otimes b'))=\psi(aa'\otimes b')=\varphi_A(a)\varphi_A(a')\varphi_B(b')
=\varphi_A(a)\psi(a'\otimes b').
\]
By linearity and continuity, we conclude that
\begin{equation}
\label{psihom1}
\psi((a\otimes 1)u)=\varphi_A(a)\psi(u)\qquad (a\in A,\; u\in A\psmash{H} B).
\end{equation}
Similarly, \eqref{a'=1} yields
\begin{equation}
\label{psihom2}
\psi(u(1\otimes b))=\psi(u)\varphi_B(b)\qquad (b\in B,\; u\in A\psmash{H} B).
\end{equation}
Since $(\varphi_A,\varphi_B)$ is a compatible pair, we have
\begin{equation}
\label{psihom3}
\psi(\tau_{B,A}(b\otimes a))=\varphi_B(b)\varphi_A(a) \qquad (b\in B,\; a\in A).
\end{equation}
Finally, by using~\eqref{prod_tau} and \eqref{psihom1}--\eqref{psihom3},
for each $a,a'\in A$ and each $b,b'\in B$
we obtain
\[
\begin{split}
\psi((a\otimes b)(a'\otimes b'))&=\psi((a\otimes 1)\tau_{B,A}(b\otimes a')(1\otimes b'))\\
&=\varphi_A(a)\varphi_B(b)\varphi_A(a')\varphi_B(b')=\psi(a\otimes b)\psi(a'\otimes b').
\end{split}
\]
Thus $\psi$ is an algebra homomorphism. The uniqueness of $\psi$ is immediate from the fact
that $\Im i_A$ and $\Im i_B$ generate a dense subalgebra of $A\psmash{H} B$.
\end{proof}

\begin{remark}
Proposition~\ref{prop:Psmash_univ} has an obvious ``purely algebraic'' version which
characterizes the algebraic smash product $A\Smash^{H} B$.
\end{remark}

Proposition~\ref{prop:Psmash_univ} motivates the following definition.

\begin{definition}
\label{def:AM_smash}
Let $H$ be a $\Ptens$-bialgebra, $A$ be a left $H$-$\Ptens$-module Arens-Michael
algebra, and $B$ be a left $H$-$\Ptens$-comodule Arens-Michael algebra.
The {\em Arens-Michael smash product} of $A$ and $B$ over $H$ is an Arens-Michael algebra
$A\Smash_{\AM}^H B$ together with a compatible pair
$(i_A\colon A\to A\Smash_{\AM}^H B,i_B \colon B\to A\Smash_{\AM}^H B)$
of continuous homomorphisms such that for each Arens-Michael algebra $C$
and each compatible pair $(\varphi_A\colon A\to C,\varphi_B\colon B\to C)$
of continuous homomorphisms there exists a unique continuous homomorphism
$A\Smash_{\AM}^H B\to C$ making the diagram
\begin{equation}
\label{AMsmash_univ}
\xymatrix@-10pt{
& A\Smash_{\AM}^H B \ar@{-->}[dd] \\
A \ar[ur]^(.4){i_A} \ar[dr]_{\varphi_A} && B \ar[ul]_(.4){i_B} \ar[dl]^{\varphi_B} \\
& C
}
\end{equation}
commute.
\end{definition}

In other words, the Arens-Michael smash product is an Arens-Michael
algebra $A\Smash_{\AM}^H B$ together with a natural isomorphism
\[
\Hom_{\AM}(A\Smash\nolimits_{\AM}^H B,C)\cong\bigl\{ \text{\upshape compatible }
(\varphi_A,\varphi_B)\in\Hom_{\AM}(A,C)\times \Hom_{\AM}(B,C)\bigr\}.
\]

Proposition~\ref{prop:Psmash_univ} easily implies that $A\Smash_{\AM}^H B$ exists.
Indeed, we have
\begin{equation}
\label{smash_AMenv}
A\Smash\nolimits_{\AM}^H B = (A\psmash{H} B)\sphat.
\end{equation}
Here is one more characterization of $A\Smash_{\AM}^H B$.

\begin{prop}
\label{prop:smash_free}
Let $H$ be a $\Ptens$-bialgebra, $A$ be a left $H$-$\Ptens$-module Arens-Michael
algebra, and $B$ be a left $H$-$\Ptens$-comodule Arens-Michael algebra. Denote by
\[
j_A\colon A\to A\Pfree B,\quad j_B\colon B\to A\Pfree B
\]
the canonical homomorphisms, and let $J$ be the closed two-sided ideal of
$A\Pfree B$ generated by
\[
\Im\bigl(\mu_{A\Pfree B}(j_B\otimes j_A-(j_A\otimes j_B)\tau_{B,A})\bigr).
\]
Finally, let $S$ be the completion of $(A\Pfree B)/J$, and let the homomorphisms
$i_A\colon A\to S$ and $i_B\colon B\to S$ be the compositions of $j_A,j_B$
with the quotient map $A\Pfree B\to S$. Then $(S,i_A,i_B)$ is the Arens-Michael
smash product of $A$ and $B$.
\end{prop}
\begin{proof}
For each Arens-Michael algebra $C$ we have natural bijections
\[
\begin{split}
\Hom_{\AM}(S,C)
&\cong\bigl\{ \varphi\in\Hom_{\AM}(A\Pfree B,C) : \varphi(J)=0\bigr\}\\
&\cong\bigl\{ \varphi\in\Hom_{\AM}(A\Pfree B,C) : \varphi \mu_{A\Pfree B}
\bigl(j_B\otimes j_A-(j_A\otimes j_B)\tau_{B,A}\bigr)=0\bigr\}\\
&\cong\bigl\{ \varphi\in\Hom_{\AM}(A\Pfree B,C) : \mu_C(\varphi\otimes\varphi)
\bigl(j_B\otimes j_A-(j_A\otimes j_B)\tau_{B,A}\bigr)=0\bigr\}\\
&\cong\bigl\{ \varphi\in\Hom_{\AM}(A\Pfree B,C) : \mu_C
\bigl(\varphi j_B\otimes \varphi j_A-(\varphi j_A\otimes \varphi j_B)\tau_{B,A}\bigr)=0\bigr\}\\
&\cong\bigl\{ \varphi\in\Hom_{\AM}(A\Pfree B,C) : (\varphi j_A,\varphi j_B)
\text{ is compatible}\bigr\}\\
&\cong\bigl\{ \text{\upshape compatible }
(\varphi_A,\varphi_B)\in\Hom_{\AM}(A,C)\times \Hom_{\AM}(B,C)\bigr\}\\
&\cong\Hom_{\AM}(A\Smash\nolimits_{\AM}^H B,C). \qedhere
\end{split}
\]
\end{proof}

\begin{corollary}
\label{cor:smash_HFG}
Let $H$ be a $\Ptens$-bialgebra, $A$ be a left $H$-$\Ptens$-module Arens-Michael
algebra, and $B$ be a left $H$-$\Ptens$-comodule Arens-Michael algebra.
If both $A$ and $B$ are holomorphically finitely generated,
then so is $A\Smash_{\AM}^H B$.
\end{corollary}
\begin{proof}
Immediate from Proposition~\ref{prop:smash_free}, Corollary~\ref{cor:HFG_free},
and Proposition~\ref{prop:HFG_quot}.
\end{proof}

So far we have two constructions of $A\Smash_{\AM}^H B$, the one given by~\eqref{smash_AMenv}
and the one given by Proposition~\ref{prop:smash_free}. However,
both of them are rather implicit.
In particular, they tell us almost nothing about the underlying locally convex space of $A\Smash_{\AM}^H B$.
Fortunately, as we will see in Proposition~\ref{prop:Psmash_stab_AM}
and Corollary~\ref{cor:smash_biarens}, we actually have
$A\Smash_{\AM}^H B=A\psmash{H} H$ under appropriate conditions.
To this end, we need some definitions and a lemma.

\begin{definition}
\label{def:S-stab}
Let $X$ be a vector space, and let $\cT$ be a set of linear operators on $X$.
We say that a seminorm $\|\cdot\|$ on $X$ is {\em $\cT$-stable} \cite{Pir_locsolv} if
for each $T\in\cT$ there exists $C>0$ such that for each $x\in X$ we have
$\| Tx\| \le C\| x\|$. A subset $U\subseteq X$
is said to be {\em $\cT$-stable} if for each $T\in\cT$ there exists $C>0$ such that
$T(U)\subseteq CU$.
\end{definition}

The following is a ``continuous version'' of Definition~\ref{def:S-stab}.

\begin{definition}
\label{def:A-stab}
Let $A$ be a locally convex algebra, and let $X$ be a left $A$-module.
We say that a seminorm $\|\cdot\|_X$ on $X$ is {\em $A$-stable} if
there exists
a continuous seminorm $\|\cdot\|_A$ on $A$ such that for all $a\in A,\; x\in X$
we have $\| a\cdot x\|_X\le \| a\|_A \| x\|_X$. A subset $U\subseteq X$
is said to be {\em $A$-stable} if there exists a $0$-neighborhood $V\subseteq A$
such that $V\cdot U\subseteq U$.
\end{definition}

\begin{remark}
Observe that a seminorm $\|\cdot\|$ is $\cT$-stable (respectively, $A$-stable)
if and only if so is the unit ball $\{ x\in X : \| x\|\le 1\}$.
\end{remark}

\begin{remark}
Clearly, each $A$-stable seminorm on $X$ is stable with respect to the set
$\{ x\mapsto a\cdot x : a\in A\}$ of multiplication operators.
If $A$ is endowed with the strongest locally convex topology,
then the converse is also true. Indeed, suppose that a seminorm
$\|\cdot\|_X$ on $X$ is stable with respect to
$\{ x\mapsto a\cdot x : a\in A\}$. Then
$\| a\|_A=\sup\{ \| a\cdot x\|_X : \| x\|_X\le 1\}$ is a seminorm on $A$
satisfying $\| a\cdot x\|_X\le \| a\|_A \| x\|_X$. Thus $\|\cdot\|_X$ is $A$-stable.
\end{remark}

\begin{definition}
\label{def:m-loc-TS}
Let $B$ be an Arens-Michael algebra, and let $\cT$ be a set of continuous linear operators
on $B$. We say that $\cT$ is {\em $m$-localizable} \cite{Pir_qfree}
if there exists a defining family of submultiplicative, $\cT$-stable
seminorms on $B$.
A linear operator $T\colon B\to B$ is {\em $m$-localizable}
if so is the singleton $\{ T\}$.
\end{definition}

Finally, we have the following ``continuous'' version of Definition~\ref{def:m-loc-TS}.

\begin{definition}
\label{def:m-loc-A}
Let $A$ be a  locally convex algebra, and let $B$ be an Arens-Michael algebra
endowed with a left $A$-module structure.
We say that the action of $A$ on $B$ is
{\em $m$-localizable} if there exists a defining family
of submultiplicative, $A$-stable seminorms on $B$.
Equivalently, the action of $A$ on $B$ is $m$-localizable if there exists a base of
idempotent, $A$-stable $0$-neighborhoods in $B$.
\end{definition}

\begin{remark}
Let $A$ be an Arens-Michael algebra, and let $B$ be an Arens-Michael algebra
endowed with a left $A$-$\Ptens$-module structure.
By~\cite[Proposition 3.4]{Pir_qfree}, there always exists a defining family of
$A$-stable seminorms on $B$. However, it is unclear whether it is possible
to find a defining family of seminorms that are $A$-stable and submultiplicative simultaneously.
\end{remark}

The following lemma (see \cite[Corollary 1.3]{Pir_locsolv})
is a special case of a useful result due to Mitiagin, Rolewicz, and
\.Zelazko \cite{MRZ}.

\begin{lemma}
\label{lemma:MRZ2}
Let $A$ be a topological algebra. Suppose that $A$ has a base
$\mathscr U$ of absolutely convex $0$-neighborhoods with the property that
for each $V\in\mathscr U$ there exist $U\in\mathscr U$ and $C>0$
such that $UV\subset CV$. Then $A$ is locally $m$-convex.
\end{lemma}

Recall a notation from the theory of topological vector spaces
(see, e.g., \cite{Groth,Kothe_II,Schaefer}).
Let $E$ and $F$ be complete locally convex spaces, and let $U\subseteq E$ and
$V\subseteq F$ be $0$-neighborhoods. By a standard abuse of notation, we
let $\ol{\Gamma(U\Tens V)}$ denote the closed absolutely convex hull
of the set $U\Tens V=\{ u\otimes v : u\in U,\; v\in V\}\subset E\Ptens F$.
If now $\cU$ (respectively, $\cV$) is a base of $0$-neighborhoods in $E$
(respectively, $F$), then $\{ \ol{\Gamma(U\Tens V)} : U\in\cU,\; V\in\cV\}$
is a base of $0$-neighborhoods in $E\Ptens F$ (loc. cit.).

Now we are in a position to show that, under appropriate conditions,
the $\Ptens$-smash product and the Arens-Michael smash product are the same.

\begin{prop}
\label{prop:Psmash_stab_AM}
Let $H$ be a $\Ptens$-bialgebra, $A$ be a left $H$-$\Ptens$-module Arens-Michael
algebra, and $B$ be a left $H$-$\Ptens$-comodule Arens-Michael algebra. Suppose that
the action of $H$ on $A$ is $m$-localizable. Then $A\psmash{H} B$
is an Arens-Michael algebra. Equivalently, $A\Smash_{\AM}^H B=A\psmash{H} B$.
\end{prop}
\begin{proof}
Let $\cU$ be a base of idempotent, $H$-stable $0$-neighborhoods in $A$,
and let $\cV$ be a base of idempotent $0$-neighborhoods in $B$.
Then the family
\[
\cB=\{ \ol{\Gamma(U\Tens V)} : U\in\cU,\; V\in\cV\}
\]
is a $0$-neighborhood base in $A\psmash{H} B$.
Given $U\in\cU$ and $V\in\cV$, let $W\subseteq H$ be a $0$-neighborhood
such that $W\cdot U\subseteq U$. Choose $V'\in\cV$
such that $\Delta_{H,B}(V')\subseteq \ol{\Gamma(W\Tens V)}$.
We claim that
\begin{equation}
\label{UV'}
\ol{\Gamma(U\Tens V')}\; \ol{\Gamma(U\Tens V)} \subseteq \ol{\Gamma(U\Tens V)}.
\end{equation}
Indeed, take $u_1,u_2\in U$, $v_1\in V'$, and $v_2\in V$.
Then $\Delta_{H,B}(v_1)\in\ol{\Gamma(W\Tens V)}$. Since $W\cdot U\subseteq U$,
i.e., $\mu_{H,A}(W\Tens U)\subseteq U$, it follows
that
\[
\begin{split}
\tau_{B,A}(v_1\otimes u_2)
&=(\mu_{H,A}\otimes\id_B)(\id_H\otimes c_{B,A})(\Delta_{H,B}(v_1)\otimes u_2)\\
&\in (\mu_{H,A}\otimes\id_B)(\id_H\otimes c_{B,A})\bigl(\ol{\Gamma(W\Tens V\Tens U)}\bigr)\\
&=(\mu_{H,A}\otimes\id_B)\bigl(\ol{\Gamma(W\Tens U\Tens V)}\bigr)
\subseteq\ol{\Gamma(U\Tens V)}.
\end{split}
\]
Therefore,
\[
\begin{split}
(u_1\otimes v_1)(u_2\otimes v_2) &=
(\mu_A\otimes\mu_B)(u_1\otimes\tau(v_1\otimes u_2)\otimes v_2)\\
& \in (\mu_A\otimes\mu_B)(U\Tens \ol{\Gamma(U\Tens V)}\Tens V)\\
&\subseteq (\mu_A\otimes\mu_B)\bigl(\ol{\Gamma(U\Tens U\Tens V\Tens V)}\bigr)
\subseteq
\ol{\Gamma(U\Tens V)}.
\end{split}
\]
This proves \eqref{UV'}. Thus the base $\cB$ satisfies the conditions of Lemma~\ref{lemma:MRZ2},
whence $A\psmash{H} B$ is an Arens-Michael algebra.
\end{proof}

\begin{prop}
\label{prop:auto_mloc}
Let $H$ be a Banach bialgebra, and let $A$ be a left $H$-$\Ptens$-module
Arens-Michael algebra. Then the action of $H$ on $A$ is $m$-localizable.
\end{prop}
\begin{proof}
Let $\{ \|\cdot\|_\lambda : \lambda\in\Lambda\}$ be a directed defining family
of submultiplicative seminorms on $A$.
Without loss of generality, we assume that the norm $\|\cdot\|$ on $H$ is submultiplicative,
and that $\| 1_H\|=1$.
Given $\lambda\in\Lambda$, find $\mu\in\Lambda$ and $C_1>0$ such that
\begin{equation}
\label{act_cont}
\| h\cdot a\|_\lambda\le C_1\| h\| \| a\|_\mu \qquad (h\in H,\; a\in A).
\end{equation}
Define a seminorm $\|\cdot\|'_\lambda$ on $A$ by
\[
\| a\|'_\lambda=\sup\{ \| h\cdot a\|_\lambda : \| h\|\le 1\}.
\]
By \eqref{act_cont}, we have $\|\cdot\|'_\lambda\le C_1\|\cdot\|_\mu$,
whence $\|\cdot\|'_\lambda$ is a continuous seminorm on $A$.
Letting $h=1$, we see that $\|\cdot\|_\lambda\le\|\cdot\|'_\lambda$.
Let now $h\in H$, $a\in A$, and assume that $\| h\|\le 1$. Then
\[
\| h\cdot a\|'_\lambda
=\sup\{ \| h_1 h\cdot a\|_\lambda : \| h_1\|\le 1\} \le \| a\|'_\lambda,
\]
due to the submultiplicativity of $\|\cdot\|$. Hence
\begin{equation}
\label{act_stab}
\| h\cdot a\|'_\lambda \le \| h\| \| a\|'_\lambda
\qquad (h\in H,\; a\in A),
\end{equation}
and so $\|\cdot\|'_\lambda$ is $H$-stable.

Let $\|\cdot\|_\pi$ denote the projective tensor norm on $H\Ptens H$, and let
$\|\cdot\|'_{\lambda,\pi}$ denote the projective tensor seminorm on $A\Ptens A$
associated to $\|\cdot\|'_\lambda$.
Observe that~\eqref{act_stab} implies the estimate
\begin{equation}
\label{act_stab_2}
\| u\cdot v\|'_{\lambda,\pi}\le \| u\|_\pi \| v\|'_{\lambda,\pi} \qquad
(u\in H\Ptens H,\; v\in A\Ptens A).
\end{equation}
Choose $C\ge 1$ such that
\begin{equation}
\label{mu_cosub}
\| \Delta_H(h)\|_\pi\le C \| h\| \qquad (h\in H).
\end{equation}
For each $a,b\in A$ we have
\begin{align*}
\| ab\|'_\lambda
&=\sup\{ \| h\cdot ab\|_\lambda : \| h\|\le 1\}\\
&=\sup\{ \| \mu_A(\Delta_H(h)\cdot (a\otimes b))\|_\lambda : \| h\|\le 1\}
&& \text{(since $\mu_A$ is an $H$-module morphism)}\\
&\le\sup\{ \| \Delta_H(h)\cdot (a\otimes b)\|_{\lambda,\pi} : \| h\|\le 1\}
&& \text{(since $\|\cdot\|_\lambda$ is submultiplicative)}\\
&\le\sup\{ \| \Delta_H(h)\cdot (a\otimes b)\|'_{\lambda,\pi} : \| h\|\le 1\}
&& \text{(since $\|\cdot\|_\lambda\le \|\cdot\|'_\lambda$)}\\
&\le\sup\{ \| \Delta_H(h)\|_\pi \| a\otimes b\|'_{\lambda,\pi} : \| h\|\le 1\}
&& \text{(by \eqref{act_stab_2})}\\
&\le C \| a\otimes b\|'_{\lambda,\pi}
&& \text{(by \eqref{mu_cosub})}\\
&=C \| a\|'_\lambda \| b\|'_\lambda.
\end{align*}
Hence $\|\cdot\|''_\lambda=C\|\cdot\|'_\lambda$ is a continuous, $H$-stable,
submultiplicative seminorm on $A$, and $\|\cdot\|_\lambda\le \|\cdot\|''_\lambda$.
Therefore $\{ \|\cdot\|''_\lambda : \lambda\in\Lambda\}$ is a directed defining family
of $H$-stable, submultiplicative seminorms on $A$, and
the action of $H$ on $A$ is $m$-localizable.
\end{proof}

\begin{corollary}
\label{cor:smash_biarens}
Let $H$ be a Banach bialgebra, $A$ be a left $H$-$\Ptens$-module
Arens-Michael algebra, and $B$ be a left $H$-$\Ptens$-comodule Arens-Michael algebra.
Then $A\psmash{H} B$
is an Arens-Michael algebra. Equivalently, $A\Smash_{\AM}^H B=A\psmash{H} B$.
\end{corollary}
\begin{proof}
Combine Propositions \ref{prop:Psmash_stab_AM} and \ref{prop:auto_mloc}.
\end{proof}

\begin{remark}
An inspection of the above proof shows that Proposition~\ref{prop:auto_mloc}
holds for each Arens-Michael bialgebra $H$ that can be represented as an inverse limit
of Banach bialgebras. Equivalently, this means that the topology on $H$ can be generated by
a family $\{ \|\cdot\|_\lambda : \lambda\in\Lambda\}$ of submultiplicative seminorms with the
additional property that $\|\Delta_H(h)\|_{\lambda,\pi}\le C_\lambda \| h\|_\lambda\; (h\in H)$,
for suitable constants $C_\lambda>0$. We restrict ourselves to Banach bialgebras
mostly because this is enough for our purposes, and also because
the vast majority of natural nonnormable Arens-Michael bialgebras do not have the above property.
We do not know, however, whether Proposition~\ref{prop:auto_mloc} holds for each
Arens-Michael bialgebra $H$.
\end{remark}

We now specialize to smash products coming from semigroup actions
and semigroup graded algebras. Given a unital semigroup $S$,
we endow the Banach algebra $\ell^1(S)$ with the structure
of a Banach bialgebra by letting
\[
\Delta(s)=s\otimes s,\quad \eps(s)=1 \qquad (s\in S).
\]
It is elementary to check that $\Delta$ and $\eps$ uniquely extend to continuous
comultiplication and counit on $\ell^1(S)$, respectively, and that $(\ell^1(S),\Delta,\eps)$ is indeed
a Banach bialgebra.

\begin{prop}
\label{prop:l1-act}
Let $S$ be a unital semigroup, and let $A$ be a $\Ptens$-algebra.
\begin{mycompactenum}
\item Suppose that $S$ acts on $A$ by endomorphisms in such a way that the
action is equicontinuous.
Then $A$ is a left $\ell^1(S)$-$\Ptens$-module algebra via the map
\begin{equation}
\label{l1_act}
\mu_{S,A}\colon\ell^1(S)\times A\to A,
\quad \Bigl(\sum_{s\in S} c_s s,a\Bigr)\mapsto\sum_{s\in S} c_s (s\cdot a).
\end{equation}
\item Conversely, if $A$ is endowed with a left $\ell^1(S)$-$\Ptens$-module algebra structure,
then the respective action of $S$ on $A$ is equicontinuous.
\end{mycompactenum}
\end{prop}
\begin{proof}
(i) Let $\{\|\cdot\|_\lambda : \lambda\in\Lambda\}$ be a directed defining family
of seminorms on $A$. Given $\lambda\in\Lambda$, find $\mu\in\Lambda$ and $C>0$ such
that $\| s\cdot a\|_\lambda\le C\| a\|_\mu$ for all $s\in S,\; a\in A$.
Then for each $h=\sum_s c_s s\in\ell^1(S)$ we have
\begin{equation}
\label{l1_act_est}
\sum_{s\in S} \| c_s (s\cdot a)\|_\lambda
\le C\sum_{s\in S} |c_s| \| a\|_\mu=C\| h\| \| a\|_\mu.
\end{equation}
Thus the family $(c_s (s\cdot a))_{s\in S}$ is absolutely summable in $A$, and~\eqref{l1_act}
yields a jointly continuous bilinear map from $\ell^1(S)\times A$ to $A$.
The axioms of an $\ell^1(S)$-$\Ptens$-module algebra are readily verified.

(ii) Let $\{\|\cdot\|_\lambda : \lambda\in\Lambda\}$ be a directed defining family
of seminorms on $A$. Given $\lambda\in\Lambda$, choose $\mu\in\Lambda$ and $C>0$
such that
\begin{equation}
\label{l1_act_est_2}
\| h\cdot a\|_\lambda\le C\| h\| \| a\|_\mu \qquad (h\in\ell^1(S),\; a\in A).
\end{equation}
Letting $h=s\in S$, we see that the action of $S$ on $A$ is equicontinuous.
\end{proof}

To introduce an appropriate ``analytic'' version of
a semigroup graded algebra, we need a definition from the theory of locally convex spaces.

\begin{definition}
\label{def:Schauder_dec}
Let $E$ be a complete locally convex space.
Following \cite[3.3]{Dineen}, we say that a family $\{ E_s : s\in S\}$ of vector subspaces of $E$
is an {\em absolute Schauder decomposition} of $E$ if the following conditions hold:
\begin{mycompactenum}
\item for each $x\in E$ there exists
a unique family $(x_s)_{s\in S}\in\prod_{s\in S} E_s$ such that $x=\sum_{s\in S} x_s$;
\item for each continuous seminorm $\|\cdot\|$ on $E$ we have
$\| x\|'=\sum_{s\in S} \| x_s\|<\infty$ (i.e., the family $(x_s)_{s\in S}$ is absolutely summable
to $x$), and, moreover, the seminorm $\|\cdot\|'$ is continuous.
\end{mycompactenum}
We say that a seminorm $\|\cdot\|$ on $E$ is {\em graded} if for each $x\in E$
we have $\| x\|=\sum_{s\in S} \| x_s\|$. It follows from the above definition that
the family of all continuous graded seminorms generates the original topology on $E$.
\end{definition}

\begin{definition}
\label{def:abs_grad}
Let $S$ be a unital semigroup, and let $A$ be a $\Ptens$-algebra. By an {\em absolute $S$-grading}
on $A$ we mean an absolute Schauder decomposition $\{ A_s : s\in S\}$ of $A$ such that
for each $s,t\in S$ we have $A_s A_t\subseteq A_{st}$.
An {\em absolutely $S$-graded $\Ptens$-algebra} is a $\Ptens$-algebra
endowed with an absolute $S$-grading.
\end{definition}

Similarly to the purely algebraic case, it is easy to see that for each $a,b\in A$ and each $r\in S$
we have $(ab)_r=\sum_{st=r} a_s b_t$. We also have $1\in A_e$, where $e$ is the identity
element of $S$ (cf. \cite[1.1.1]{Nast_methods}).

\begin{prop}
\label{prop:l1-coact}
Let $S$ be a unital semigroup, and let $A$ be a $\Ptens$-algebra.
\begin{mycompactenum}
\item If $A$ is endowed with an absolute $S$-grading,
then $A$ is a left $\ell^1(S)$-$\Ptens$-comodule algebra via the map
\begin{equation}
\label{l1_coact}
\Delta_{S,A}\colon A\to\ell^1(S)\Ptens A,
\quad a\mapsto\sum_{s\in S} s\otimes a_s \qquad (a\in A).
\end{equation}
\item Conversely, if $A$ is endowed with a left $\ell^1(S)$-$\Ptens$-comodule algebra structure,
then the subspaces $A_s=\{ a\in A : \Delta_{S,A}(a)=s\otimes a\}$ form an absolute
$S$-grading on $A$.
\end{mycompactenum}
\end{prop}
\begin{proof}
Let $\{ \|\cdot\|_\lambda : \lambda\in\Lambda\}$
be a directed defining family of
seminorms on $A$. For each $\lambda\in\Lambda$, we will use the same notation
$\|\cdot\|_\lambda$ for the respective projective tensor seminorm on $\ell^1(S)\Ptens A$.

(i) Suppose that $A$ is absolutely $S$-graded. Without loss of generality, we may assume that
each seminorm $\|\cdot\|_\lambda$ is graded. For each $a\in A$, we have
\[
\sum_{s\in S} \| s\otimes a_s\|_\lambda=\sum_{s\in S} \| a_s\|_\lambda=\| a\|_\lambda.
\]
Hence the family $(s\otimes a_s)_{s\in S}$ is absolutely summable in $\ell^1(S)\Ptens A$,
and~\eqref{l1_coact} yields a continuous linear map from $A$ to $\ell^1(S)\Ptens A$.
The axioms of a left $\ell^1(S)$-$\Ptens$-comodule
algebra are readily verified.

(ii) Conversely, assume that $A$ is a left $\ell^1(S)$-$\Ptens$-comodule algebra.
Recall \cite[7.2.3]{Pietsch_nucl} that there exists a topological isomorphism between $\ell^1(S)\Ptens A$
and the locally convex space $\ell^1(S,A)$ of all absolutely summable families $(a_s)_{s\in S}$
in $A$. Explicitly, the isomorphism takes each $(a_s)_{s\in S}\in\ell^1(S,A)$
to $\sum_{s\in S} s\otimes a_s\in\ell^1(S)\Ptens A$,
and for all $\lambda\in\Lambda$ we have
$\| \sum_{s\in S} s\otimes a_s\|_\lambda=\sum_{s\in S} \| a_s\|_\lambda$.
Hence for each $a\in A$ there
exists a unique absolutely summable family $(a_s)_{s\in S}$ in $A$ such that
the coaction $\Delta_{S,A}$ of $\ell^1(S)$ on $A$ is given by~\eqref{l1_coact}.
By coassociativity, we have
\[
\sum_{s\in S} s\otimes\Delta_{S,A}(a_s)=\sum_{s\in S} s\otimes s\otimes a_s,
\]
whence $\Delta_{S,A}(a_s)=s\otimes a_s$, i.e., $a_s\in A_s$.
We also have
\[
a=(\eps\otimes\id_A)(\Delta_{S,A}(a))=\sum_{s\in S} a_s\qquad (a\in A).
\]
If $(a'_s)_{s\in S}\in\prod_s A_s$ is another summable family
such that $a=\sum_{s\in S} a'_s$, then applying $\Delta_{S,A}$ yields
$\sum_{s\in S} s\otimes a_s=\sum_{s\in S} s\otimes a'_s$, whence $a'_s=a_s$ for all $s\in S$.
Thus condition (i) of Definition~\ref{def:Schauder_dec} is satisfied.
Clearly, for each $\lambda\in\Lambda$ the seminorm $\| a\|'_\lambda=\|\Delta_{S,A}(a)\|_\lambda$
is continuous on $A$. On the other hand,
$\| a\|'_\lambda=\sum_{s\in S} \| a_s\|_\lambda$.
Thus condition (ii) of Definition~\ref{def:Schauder_dec} is
also satisfied, and so $\{ A_s : s\in S\}$ is an absolute Schauder decomposition of $A$.
Finally, for each $a\in A_s$ and $b\in A_t$ we have
$\Delta_{S,A}(ab)=(s\otimes a)(t\otimes b)=st\otimes ab$, i.e., $ab\in A_{st}$.
This completes the proof.
\end{proof}

\begin{theorem}
\label{thm:S-smash}
Let $A$ be a $\Ptens$-algebra. Suppose that a unital semigroup $S$ acts on $A$
by endomorphisms in such a way that the action is equicontinuous.
Let also $B$ be an absolutely $S$-graded $\Ptens$-algebra.
Then there exists a unique $\Ptens$-algebra structure on $A\Ptens B$ such that
\begin{align}
\label{smash_S_1}
(a_1\otimes 1)(a_2\otimes 1)&=a_1a_2\otimes 1 && (a_1,a_2\in A);\\
(1\otimes b_1)(1\otimes b_2)&=1\otimes b_1 b_2 && (b_1,b_2\in B);\\
\label{smash_S_3}
(a\otimes 1)(1\otimes b)&=a\otimes b && (a\in A,\; b\in B);\\
\label{smash_S_4}
(1\otimes b_s)(a\otimes 1)&=(s\cdot a)\otimes b_s && (a\in A,\; b_s\in B_s,\; s\in S).
\end{align}
The resulting $\Ptens$-algebra $A\psmash{S} B$ is equal to $A\psmash{\ell^1(S)} B$.
If, in addition, $A$ and $B$ are Arens-Michael algebras, then so is $A\psmash{S} B$.
Finally, if both $A$ and $B$ are holomorphically finitely generated, then so is $A\psmash{S} B$.
\end{theorem}
\begin{proof}
Applying Propositions~\ref{prop:l1-act} and~\ref{prop:l1-coact}, we obtain a
left $\ell^1(S)$-$\Ptens$-module algebra structure on $A$ and a left
$\ell^1(S)$-$\Ptens$-comodule algebra structure on $B$.
Let $A\psmash{S} B = A\psmash{\ell^1(S)} B$. Relations~\eqref{smash_S_1}--\eqref{smash_S_4}
are immediate from Lemma~\ref{lemma:id_smash} modulo the fact that
$\Delta_{S,B}(b_s)=s\otimes b$.
By~\eqref{smash_S_3}, the elements $a\otimes 1\; (a\in A)$
and $1\otimes b_s\; (s\in S,\; b_s\in B_s)$ generate a dense subalgebra of $A\psmash{S} B$,
whence the multiplication on $A\psmash{S} B$ is uniquely
determined by~\eqref{smash_S_1}--\eqref{smash_S_4}.
If both $A$ and $B$ are Arens-Michael algebras, then
Corollary~\ref{cor:smash_biarens} implies
that $A\psmash{S} B$ is an Arens-Michael algebra too.
The last assertion follows from Corollary~\ref{cor:smash_HFG}.
\end{proof}

\section{Skew holomorphic functions}
\label{sect:skew}

A construction closely related to smash products is that of an Ore extension, or a skew
polynomial ring (see, e.g., \cite{MR_book,Kassel}).
Similarly to free products and smash products, Ore extensions also have natural
``analytic'' counterparts~\cite{Pir_qfree}.

Let $A$ be an algebra, and let $\sigma$ be an endomorphism of $A$. Recall that a linear map
$\delta\colon A\to A$ is an {\em $\sigma$-derivation} if $\delta(ab)=\delta(a)b+\sigma(a)\delta(b)$
for all $a,b\in A$. Suppose now that $A$ is an Arens-Michael algebra, $\sigma$ is an endomorphism
of $A$, and $\delta$ is a $\sigma$-derivation such that $\{ \sigma,\delta\}$ is $m$-localizable.
By~\cite[Proposition 4.3]{Pir_qfree}, there exists a unique continuous multiplication
on $\cO(\CC,A)$ such that the canonical embeddings
\begin{equation}
\label{emb_holvect}
\begin{split}
\cO(\CC)&\hookrightarrow\cO(\CC,A),\quad f\mapsto f\otimes 1,\\
A&\hookrightarrow\cO(\CC,A), \quad a\mapsto 1\otimes a,
\end{split}
\end{equation}
are algebra homomorphisms, and such that
\begin{equation}
\label{Ore_relation}
za=\sigma(a)z+\delta(a) \qquad (a\in A),
\end{equation}
where $z\in\cO(\CC)\subset\cO(\CC,A)$ is the complex coordinate.
The resulting $\Ptens$-algebra is denoted by $\cO(\CC,A;\sigma,\delta)$ and is called the
{\em Arens-Michael Ore extension} (or the {\em analytic Ore extension}) of $A$
via $\{\sigma,\delta\}$.
By~\cite[Proposition 4.5]{Pir_qfree}, $\cO(\CC,A;\sigma,\delta)$ is
an Arens-Michael algebra.

By~\cite[Proposition 4.4 and Remark 4.6]{Pir_qfree},
$\cO(\CC,A;\sigma,\delta)$ has the universal property that
for each Arens-Michael algebra $B$, each continuous homomorphism $\varphi\colon A\to B$,
and each $x\in B$ satisfying
$x\varphi(a)=\varphi(\sigma(a)) x+\varphi(\delta(a))\; (a\in A)$,
there exists a unique continuous homomorphism $\psi\colon\cO(\CC,A;\sigma,\delta)\to B$
such that $\psi|_A=\varphi$ and $\psi(z)=x$. In other words, there is a natural bijection
\begin{equation}
\label{Ore_univ}
\begin{split}
\Hom_{\AM}(\cO(\CC,A;\sigma,\delta),B)
\cong \bigl\{ (\varphi,x)&\in\Hom_{\AM}(A,B)\times B :\\
&x\varphi(a)=\varphi(\sigma(a)) x+\varphi(\delta(a))\;\forall a\in A\bigr\}.
\end{split}
\end{equation}

\begin{remark}
The algebra $\cO(\CC,A;\sigma,\delta)$ can be viewed as an analytic analog of
the algebraic Ore extension $A[z;\sigma,\delta]$ (see, e.g., \cite[1.7]{Kassel}).
Note that $A[z;\sigma,\delta]$ is a dense subalgebra of $\cO(\CC,A;\sigma,\delta)$.
In fact, it is rather easy to deduce from~\eqref{Ore_univ} that, if we endow
$A[z;\sigma,\delta]$ with a suitable topology, then
$\cO(\CC,A;\sigma,\delta)$ becomes the Arens-Michael envelope of $A[z;\sigma,\delta]$.
\end{remark}

\begin{prop}
\label{prop:Ore_free}
Let $A$ be an Arens-Michael algebra, $\sigma$ be an endomorphism
of $A$, and $\delta$ be a $\sigma$-derivation on $A$ such that
$\{ \sigma,\delta\}$ is $m$-localizable.
Identify $\cO(\CC)$ and $A$ with the respective
subalgebras of $\cO(\CC)\Pfree A$ via the
canonical embeddings $j_{\cO(\CC)}$ and $j_A$, and let
$J$ be the closed two-sided ideal of
$\cO(\CC)\Pfree A$ generated by
\[
\{ za-\sigma(a)z-\delta(a) : a\in A\}.
\]
Finally, let $S$ be the completion of $(\cO(\CC)\Pfree A)/J$. Then there exists
a unique topological algebra isomorphism
$S\cong\cO(\CC,A;\sigma,\delta)$ such that $z+J\mapsto z$ and $a+J\mapsto a$ for all $a\in A$.
\end{prop}

We omit the proof since it is similar to that of Proposition~\ref{prop:smash_free}
modulo~\eqref{Ore_univ} and the natural bijection $\Hom_{\AM}(\cO(\CC),B)\cong B\; (B\in\AM)$.

\begin{corollary}
Let $A$ be an Arens-Michael algebra, $\sigma$ be an endomorphism
of $A$, and $\delta$ be a $\sigma$-derivation on $A$ such that
$\{ \sigma,\delta\}$ is $m$-localizable. If $A$ is holomorphically finitely generated,
then so is $\cO(\CC,A;\sigma,\delta)$.
\end{corollary}

If $\sigma=\id_A$ (respectively, if $\delta=0$), then the algebra $\cO(\CC,A;\sigma,\delta)$
is denoted by $\cO(\CC,A;\delta)$ (respectively, $\cO(\CC,A;\sigma)$).
These algebras can be interpreted as Arens-Michael smash products as follows
(for details, see \cite[Remarks~4.2--4.4]{Pir_qfree}).
Observe that the Fr\'echet-Arens-Michael algebra $\cO(\CC)$ can be made into
a $\Ptens$-bialgebra in two different ways. The first way is to use the additive
structure on $\CC$ and to define a comultiplication $\Delta_{\add}$ and a counit
$\eps_{\add}$ on $\cO(\CC)$ by
\[
\Delta_{\add}(f)(z,w)=f(z+w),\quad \eps_{\add}(f)=f(0)\qquad (f\in\cO(\CC)).
\]
The resulting Arens-Michael bialgebra (which is in fact an Arens-Michael Hopf algebra)
will be denoted by $\cO(\CC_{\add})$. Alternatively, we can use the multiplicative structure
on $\CC$ to define a comultiplication $\Delta_{\mult}$ and a counit
$\eps_{\mult}$ on $\cO(\CC)$ by
\[
\Delta_{\mult}(f)(z,w)=f(zw),\quad \eps_{\mult}(f)=f(1)\qquad (f\in\cO(\CC)).
\]
The resulting Arens-Michael bialgebra will be denoted by $\cO(\CC_{\mult})$.
If now $A$ is an Arens-Michael algebra and $\delta$ is an $m$-localizable derivation on $A$,
then $A$ can be made into an $\cO(\CC_{\add})$-$\Ptens$-module algebra
in such a way that $z\cdot a=\delta(a)\; (a\in A)$. Moreover, the action of $\cO(\CC_{\add})$
on $A$ is easily seen to be $m$-localizable, and we have topological algebra
isomorphisms
\begin{equation}
\label{Ore_smash_1}
\cO(\CC,A;\delta)\cong A\Psmash\cO(\CC_{\add}) \cong A\Smash\nolimits_{\AM} \cO(\CC_{\add}).
\end{equation}
Similarly, if $\sigma$ is an $m$-localizable endomorphism of $A$,
then $A$ can be made into an $\cO(\CC_{\mult})$-$\Ptens$-module algebra
in such a way that $z\cdot a=\sigma(a)\; (a\in A)$. Moreover, the action of $\cO(\CC_{\mult})$
on $A$ is easily seen to be $m$-localizable, and we have topological algebra
isomorphisms
\begin{equation}
\label{Ore_smash_2}
\cO(\CC,A;\sigma)\cong A\Psmash\cO(\CC_{\mult}) \cong A\Smash\nolimits_{\AM} \cO(\CC_{\mult}).
\end{equation}
Of course, the isomorphisms \eqref{Ore_smash_1} and \eqref{Ore_smash_2}
come from their algebraic versions
$A[z;\delta]\cong A\Smash\cO^\reg(\CC_\add)\cong A\Smash U(\CC)$
(where $U(\CC)$ is the enveloping algebra of the abelian Lie algebra $\CC$) and
$A[z;\sigma]\cong A\Smash\cO^\reg(\CC_\mult)\cong A\Smash\CC\Z_+$
(where $\CC\Z_+$ is the semigroup algebra of the additive semigroup $\Z_+$).

Our next goal is to extend the definition of $\cO(\CC,A;\sigma)$ to the case
of $A$-valued functions on a domain $D\subseteq\CC^n$ satisfying some additional
conditions.
Recall that a subset $D\subseteq\CC^n$ is {\em balanced} if for each $\lambda\in\CC$
with $|\lambda|\le 1$ we have $\lambda D\subseteq D$. Given a balanced domain
$D\subseteq\CC^n$ and $f\in\cO(D)$, let $f_m$ denote the homogeneous polynomial
of total degree $m\in\Z_+$ that appears in the Taylor expansion of $f$ at $0$.
Explicitly,
\[
f_m(z)=\sum_{\substack{k\in\Z_+^n\\ |k|=m}} \frac{D^k f(0)}{k!} z^k.
\]
By \cite[Proposition 3.36]{Dineen},  the series $\sum_{m\in\Z_+}  f_m$ absolutely converges
in $\cO(D)$, and we have $f=\sum_{m\in\Z_+} f_m$.
Moreover (loc. cit.), for each compact set $K\subset D$ the seminorm
\[
\| f\|'_K=\sum_{m\in\Z_+} \| f_m\|_K \qquad (\text{where}\; \| f_m\|_K=\sup_{z\in K} |f_m(z)|)
\]
is continuous on $\cO(D)$. Thus we see that $\cO(D)$ becomes an absolutely $\Z_+$-graded
Fr\'echet algebra (see Definition~\ref{def:abs_grad}).

\begin{prop}
\label{prop:skew_hol_bal}
Let $A$ be a $\Ptens$-algebra, and let $\sigma$ be an endomorphism of $A$
such that the family $\{ \sigma^k : k\in\Z_+\}$ is equicontinuous.
Then for each balanced domain $D\subseteq\CC^n$
there exists a unique $\Ptens$-algebra structure on $\cO(D,A)$ such that
the canonical embeddings
\begin{equation*}
\begin{split}
\cO(D)&\hookrightarrow\cO(D,A),\quad f\mapsto f\otimes 1,\\
A&\hookrightarrow\cO(D,A), \quad a\mapsto 1\otimes a,
\end{split}
\end{equation*}
are algebra homomorphisms, and such that
\begin{equation}
\label{rel_skew_hol_bal}
z_i a=\sigma(a)z_i \qquad (a\in A,\; i=1,\ldots ,n).
\end{equation}
The resulting $\Ptens$-algebra $\cO(D,A;\sigma)$ is equal to
$A\psmash{\Z_+}\cO(D)$.
If, in addition, $A$ is an Arens-Michael algebra, then so is $\cO(D,A;\sigma)$.
Finally, if $A$ is holomorphically finitely generated, then so is $\cO(D,A;\sigma)$.
\end{prop}
\begin{proof}
We have an equicontinuous action of $\Z_+$ on $A$ given by
$k\cdot a=\sigma^k(a)\; (k\in\Z_+,\; a\in A)$.
As we have already observed, $\cO(D)$ is an absolutely $\Z_+$-graded Fr\'echet algebra.
Now the result follows from Theorem~\ref{thm:S-smash} modulo the obvious fact that
$A$ and $z_1,\ldots ,z_n$ generate a dense subalgebra of $\cO(D,A;\sigma)$.
\end{proof}

Proposition~\ref{prop:skew_hol_bal} can be extended to the case of several commuting
endomorphisms of $A$. To this end, we have to impose more restrictions on the
domain $D$. Recall that a domain $D\subseteq\CC^n$ is a {\em Reinhardt domain}
if for each $z=(z_1,\ldots ,z_n)\in D$ and each $\lambda=(\lambda_1,\ldots ,\lambda_n)\in\CC^n$
such that $|\lambda_i|=1\; (i=1,\ldots ,n)$ we have $(\lambda_1 z_1,\ldots ,\lambda_n z_n)\in D$.
If the above condition holds for all $\lambda=(\lambda_1,\ldots ,\lambda_n)\in\CC^n$
such that $|\lambda_i|\le 1\; (i=1,\ldots ,n)$, then $D$ is a {\em complete Reinhardt domain}.
Equivalently, $D$ is a Reinhardt domain (respectively, a complete Reinhardt domain)
if and only if $D$ is a union of closed polyannuli (respectively, of closed polydisks)
centered at $0$.

Given a Reinhardt domain $D\subseteq\CC^n$, let
\[
N_D=\bigl\{ j\in\{1,\ldots ,n\} : D\cap\{ z\in\CC^n : z_j=0\}\ne\varnothing\bigr\}.
\]
Note that if $0\in D$ (e.g., if $D$ is complete), then $N_D=\{ 1,\ldots ,n\}$.
Clearly, the monomial $z^k$ (where $k\in\Z^n$) is defined everywhere on $D$
if and only if $k_j\ge 0$ for all $j\in N_D$.

The following is essentially a restatement of a well-known fact from complex analysis
(see, e.g., \cite[Chap.~II, \S1.2]{Range}).

\begin{lemma}
For each Reinhardt domain $D\subseteq\CC^n$, the subspaces
\[
\cO(D)_k=
\begin{cases}
\spn\{ z^k\} & \text{if $k_j\ge 0$ for all $j\in N_D$};\\
0 & \text{otherwise}
\end{cases}
\qquad (k\in\Z^n)
\]
form an absolute $\Z^n$-grading on $\cO(D)$. In particular, if $0\in D$, then
$\cO(D)$ is absolutely $\Z_+^n$-graded.
\end{lemma}
\begin{proof}
By \cite[Chap.~II, \S1.2, Theorem 1.5]{Range}, each $f\in\cO(D)$
has a unique Laurent series representation
\[
f(z)=\sum_{k\in\Z^n} c_k z^k \qquad (z\in D),
\]
and $c_k=0$ whenever $k_j<0$ for some $j\in N_D$. Thus for each $k\in\Z^n$
we have $f_k=c_k z^k\in\cO(D)_k$. Fix a compact polyannulus $\bar\DD^n_{r,R}\subset D$,
and let $\lambda\in (0,1)$ be such that $\bar\DD_{\lambda r,\lambda^{-1}R}\subset D$.
By~\cite[Chap.~II, \S1.2, (1.11)]{Range}, we have
\[
\| f_k \|_{\bar\DD_{r,R}} \le \lambda^{|k|} \| f\|_{\bar\DD_{\lambda r,\lambda^{-1}R}}
\qquad (k\in\Z^n).
\]
Therefore
\begin{equation}
\label{Rnh_grad_est}
\sum_{k\in\Z^n} \| f_k \|_{\bar\DD_{r,R}} \le C \| f\|_{\bar\DD_{\lambda r,\lambda^{-1}R}},
\end{equation}
where
\[
C=\sum_{k\in\Z^n} \lambda^{|k|}=\left(\frac{1+\lambda}{1-\lambda}\right)^n.
\]
Since each compact set $K\subset D$ is covered by finitely many compact polyannuli
contained in $D$, it follows from~\eqref{Rnh_grad_est}
that the seminorm $\| f\|_K''=\sum_k \| f_k\|_K$
is continuous on $\cO(D)$. This completes the proof.
\end{proof}

\begin{prop}
\label{prop:skew_hol_Rnhrdt}
Let $A$ be a $\Ptens$-algebra, and let $\sigma=(\sigma_1,\ldots ,\sigma_n)$ be a
commuting family of endomorphisms of $A$. Suppose that $D\subseteq\CC^n$ is a Reinhardt domain,
and that one of the following conditions holds:
\begin{mycompactenum}
\item $0\in D$, and the subsemigroup of $\End(A)$ generated by $\sigma_1,\ldots ,\sigma_n$
is equicontinuous;
\item $\sigma_1,\ldots ,\sigma_n$ are automorphisms, and the subgroup of $\Aut(A)$
generated by $\sigma_1,\ldots ,\sigma_n$ is equicontinuous.
\end{mycompactenum}
Then there exists a unique $\Ptens$-algebra structure on $\cO(D,A)$ such that
the canonical embeddings
\begin{equation*}
\begin{split}
\cO(D)&\hookrightarrow\cO(D,A),\quad f\mapsto f\otimes 1,\\
A&\hookrightarrow\cO(D,A), \quad a\mapsto 1\otimes a,
\end{split}
\end{equation*}
are algebra homomorphisms, and such that
\begin{equation}
\label{rel_skew_hol_Rnhrdt}
z_i a=\sigma_i(a)z_i \qquad (a\in A,\; i=1,\ldots ,n).
\end{equation}
Let $S=\Z_+^n$ for case {\upshape (i)}, and $S=\Z^n$ for case {\upshape (ii)}. Then
the resulting $\Ptens$-algebra $\cO(D,A;\sigma)$ is equal to $A\psmash{S}\cO(D)$.
If, in addition, $A$ is an Arens-Michael algebra, then so is $\cO(D,A;\sigma)$.
Finally, if $A$ is holomorphically finitely generated, then so is $\cO(D,A;\sigma)$.
\end{prop}

We omit the proof as it is similar to that of Proposition~\ref{prop:skew_hol_bal}.

\begin{remark}
In practice, the equicontinuity of the subsemigroup of $\End(A)$ generated by
commuting endomorphisms $\sigma_1,\ldots ,\sigma_n$ is often reduced to the following
property. Suppose that there exists a defining family $\{ \|\cdot\|_\lambda : \lambda\in\Lambda\}$
of seminorms on $A$ such that
\begin{equation}
\label{suff_equi_semi}
\| \sigma_i(a)\|_\lambda \le \| a\|_\lambda \qquad (a\in A,\; \lambda\in\Lambda,\; i=1,\ldots ,n).
\end{equation}
Then for each $k\in\Z_+^n$ and each $a\in A$ we have $\| \sigma^k(a)\|_\lambda \le \| a\|_\lambda$,
which obviously implies that the subsemigroup of $\End(A)$ generated by $\sigma_1,\ldots ,\sigma_n$
is equicontinuous. Similarly, if $\sigma_1,\ldots ,\sigma_n$ are automorphisms,
then the subgroup of $\Aut(A)$ generated by $\sigma_1,\ldots ,\sigma_n$
is equicontinuous provided that
there exists a defining family $\{ \|\cdot\|_\lambda : \lambda\in\Lambda\}$
of seminorms on $A$ such that
\begin{equation}
\label{suff_equi_grp}
\| \sigma_i(a)\|_\lambda = \| a\|_\lambda \qquad
(a\in A,\; \lambda\in\Lambda,\; i=1,\ldots ,n).
\end{equation}
\end{remark}

\begin{remark}
Algebras $\cO(D,A;\sigma)$ from Proposition~\ref{prop:skew_hol_Rnhrdt}~(i)
can be viewed as analytic analogs of iterated Ore extensions of the following special form.
Given an algebra $A$ and a commuting $n$-tuple $\sigma=(\sigma_1,\ldots ,\sigma_n)$
of endomorphisms of $A$, we define a chain $A_0\subset\cdots\subset A_n$
of algebras by letting $A_0=A$ and $A_k=A_{k-1}[z_k;\tilde\sigma_k]$, where $\tilde\sigma_k$
is the endomorphism of $A_{k-1}$ uniquely determined by
$\tilde\sigma_k|_A=\sigma_k$ and $\tilde\sigma_k(z_i)=z_i\; (i=1,\ldots ,k-1)$.
The existence of such an endomorphism easily follows from the commutativity of the $\sigma_i$'s.
We have $A_n=A[z_1,\ldots ,z_n]$ with multiplication uniquely determined by~\eqref{rel_skew_hol_Rnhrdt}
and by the requirement that the embeddings
$A\hookrightarrow A_n$ and $\CC[z_1,\ldots ,z_n]\hookrightarrow A_n$ are algebra homomorphisms.
If now $A$, $\sigma$, and $D$ are such as in Proposition~\ref{prop:skew_hol_Rnhrdt}~(i),
then $A_n$ is obviously a dense subalgebra of $\cO(D,A;\sigma)$.
If condition (ii) of Proposition~\ref{prop:skew_hol_Rnhrdt} holds, and if $N_D=\varnothing$
(e.g., if $D$ is a polyannulus), then the situation is similar, but one should use
iterated Laurent extensions instead of iterated Ore extensions.

In the special case where $D=\DD_R^n$ is a polydisk in $\CC^n$,
the algebra $\cO(D,A;\sigma)$ can also be defined by a similar iterative procedure.
Specifically, we have a chain $B_0\subset\cdots\subset B_n$
of Fr\'echet algebras with $B_0=A$ and $B_k=\cO(\DD_{R_k},B_{k-1};\tilde\sigma_k)$,
where
\[
\tilde\sigma_k=\sigma_k\otimes\id\colon
A\Ptens\cO(\DD_{R_1}\times\cdots\times\DD_{R_{k-1}})
\to A\Ptens\cO(\DD_{R_1}\times\cdots\times\DD_{R_{k-1}}).
\]
It is easy to show that $B_n\cong\cO(D,A;\sigma)$. A similar picture holds in the case
where $D$ is a polyannulus.
\end{remark}

\section{Examples}
\label{sect:examples}

\subsection{Arens-Michael envelopes}
\label{subsect:AM}

Let $A$ be a finitely generated algebra, and suppose that $A$ is good enough to be interpreted as the
algebra of ``regular functions'' on a ``noncommutative affine variety''.
According to the point of view adopted in \cite{Pir_qfree}, the Arens-Michael envelope of $A$
is then a natural candidate for the algebra of ``holomorphic
functions'' on the same ``variety''. This agrees with the philosophy of the present paper,
as we will now see.

\begin{prop}
\label{prop:AM_HFG}
If $A$ is a finitely generated algebra, then $\wh{A}$ is holomorphically finitely generated.
\end{prop}
\begin{proof}
Let $A=F_n/I$, where $I\subset F_n$ is a two-sided ideal.
Applying~\cite[Corollary 3.2]{Pir_qfree}, we conclude that
$\wh{A}=\cF_n/J$, where $J$ is the closure of $I$ in $\cF_n$.
By Proposition~\ref{prop:HFG_quot_F}, $\wh{A}$ is an HFG algebra.
\end{proof}

We have already noticed in Section~\ref{sect:prelim}
that, according to \cite{T2}, the Arens-Michael envelope of the polynomial algebra $\CC[z_1,\ldots ,z_n]$
is the algebra of entire functions $\cO(\CC^n)$, while
the Arens-Michael envelope of the free algebra $F_n$ is the algebra of free entire functions $\cF_n$.
The following example is a slight generalization of \cite[Example 3.6]{Pir_qfree}

\begin{example}
\label{ex:AM_O}
Let $(X,\cO^\reg_X)$ be an affine scheme of finite type over $\CC$, and
let $(X_h,\cO_{X_h})$ be the complex space associated to $(X,\cO^\reg_X)$
(cf. \cite[Appendix~B]{Hart_AG}). We claim that the Arens-Michael envelope of
$A=\cO^\reg(X)$ is $\cO(X_h)$. Indeed, we have $A=\cO^\reg(\CC^n)/I$
for an ideal $I\subset\cO^\reg(\CC^n)$.
Let $f_1,\ldots ,f_q\in\cO^\reg(\CC^n)$ generate $I$; then, by
definition~\cite[Appendix~B]{Hart_AG}, we have $\cO_{X_h}=\cO_{\CC^n}/\cI$,
where $\cI\subset\cO_{\CC^n}$ is the ideal sheaf generated by $f_1,\ldots ,f_q$.
Thus we have an exact sequence
\[
\cO_{\CC^n}^q \xra{(f_1,\ldots ,f_q)} \cO_{\CC^n} \to \cO_{X_h} \to 0
\]
of $\cO_{\CC^n}$-modules. Applying the global section functor, we see that
$\cO(X_h)\cong\cO(\CC^n)/J$, where $J\subset\cO(\CC^n)$ is the ideal generated
by $f_1,\ldots ,f_q$ (note that $J$ is automatically closed in $\cO(\CC^n)$;
see \cite[V.6]{GR_II}). We claim that $J=\ol{I}$, the closure of $I$ in $\cO(\CC^n)$.
Indeed, since $J$ is closed, we have $\ol{I}\subseteq J$. On the other hand,
$\ol{I}$ is an ideal of $\cO(\CC^n)$, because $\cO^\reg(\CC^n)$ is dense in $\cO(\CC^n)$.
Since $J$ is algebraically generated by $f_1,\ldots ,f_q$, we see that $J\subseteq\ol{I}$,
and finally $J=\ol{I}$. By using the fact that the Arens-Michael functor commutes with
quotients~\cite[Corollary 3.2]{Pir_qfree}, we conclude that
\[
\wh{\cO^\reg(X)}
\cong \wh{\cO^\reg(\CC^n)}/\ol{I} \cong \cO(\CC^n)/J \cong \cO(X_h).
\]
\end{example}

Thus we may interpret the Arens-Michael functor as a noncommutative extension
of the functor that takes an affine scheme $(X,\cO^\reg_X)$ of finite type over $\CC$
to the associated complex space $(X_h,\cO_{X_h})$.

\begin{remark}
\label{rem:non-HFG_closed}
The following example shows that a closed subalgebra of an HFG algebra need not be
an HFG algebra. Let $\fg$ be the two-dimensional solvable Lie algebra with basis $\{ x,y\}$
and commutation relation $[x,y]=y$. As was shown in \cite[Proposition 5.2]{Pir_qfree}, the Arens-Michael
envelope of the enveloping algebra $U(\fg)$ is given by
\begin{equation}
\label{AM_U}
\wh{U}(\fg)=
\Bigl\{ a=\sum_{i,j=0}^\infty c_{ij} x^i y^j :
\| a\|_{n,t}=\sum_{i=0}^\infty\sum_{j=0}^n |c_{ij}| t^i \quad \forall t>0,\;\forall n\in\N\Bigr\}.
\end{equation}
The topology on $\wh{U}(\fg)$ is given by the seminorms $\|\cdot\|_{n,t}\; (t>0,\; n\in\N)$.
By Proposition~\ref{prop:AM_HFG}, $\wh{U}(\fg)$ is an HFG algebra.
On the other hand, it immediate from~\eqref{AM_U} that the closed subalgebra of $\wh{U}(\fg)$
generated by $y$ is the algebra $\CC[[y]]$ of formal power series. Clearly,
$\CC[[y]]$ is not isomorphic to $\cO(X)$ for any Stein space $(X,\cO_X)$
and hence is not an HFG algebra.
\end{remark}

The following two examples \cite{Pir_qfree} will motivate our further constructions,
so we reproduce them here.

\begin{example}
\label{ex:q_aff}
Let $\bq=(q_{ij})$ be a complex $n\times n$-matrix such that $q_{ii}=1$ and $q_{ji}=q_{ij}^{-1}$
for all $i,j=1,\ldots ,n$ (such matrices are called {\em multiplicatively antisymmetric}).
Recall that the algebra $\cO_{\bq}^\reg(\CC^n)$ {\em of regular functions on the
quantum affine $n$-space} is generated by
$n$ elements $x_1,\ldots ,x_n$ subject to the relations $x_i x_j=q_{ij}x_j x_i$ for all $i,j$
(see, e.g., \cite{Br_Good}).
If $q_{ij}=1$ for all $i,j$, then $\cO_\bq^\reg(\CC^n)$ is nothing but the
polynomial algebra $\CC[x_1,\ldots ,x_n]=\cO^\reg(\CC^n)$.
Of course, $\cO_\bq^\reg(\CC^n)$ is noncommutative unless $q_{ij}=1$ for all $i,j$,
but the monomials
$x^k=x_1^{k_1}\cdots x_n^{k_n}\; (k\in\Z_+^n)$ still
form a basis of $\cO_\bq^\reg(\CC^n)$. Thus $\cO_\bq^\reg(\CC^n)$ may be viewed as
a ``deformed'' polynomial algebra.

The Arens-Michael envelope of $\cO_\bq^\reg(\CC^n)$
is denoted by $\cO_\bq(\CC^n)$ and is called the
{\em algebra of holomorphic functions on the quantum affine $n$-space}.
If $q_{ij}=1$ for all $i,j$, then $\cO_\bq(\CC^n)\cong\cO(\CC^n)$
(see Section~\ref{sect:prelim} or Example~\ref{ex:AM_O}).
The algebra $\cO_\bq(\CC^n)$ has the following explicit description \cite{Pir_qfree}.
Given $d\in\Z_+$, let us identify each element $\alpha=(\alpha_1,\ldots ,\alpha_d)\in W_{n,d}$
with the function $\alpha\colon \{ 1,\ldots ,d\}\to\{ 1,\ldots ,n\},\; \alpha(i)=\alpha_i$.
The symmetric group $S_d$ acts on $W_{n,d}$ via
$\sigma(\alpha)=\alpha\sigma^{-1}\; (\alpha\in W_{n,d},\; \sigma\in S_d)$.
Clearly, for each $\alpha\in W_{n,d}$ and $\sigma\in S_d$ there exists a unique
$\lambda(\sigma,\alpha)\in\CC^\times$ (where $\CC^\times=\CC\setminus\{ 0\}$)
such that
\begin{equation}
\label{lambda}
x_\alpha=\lambda(\sigma,\alpha)x_{\sigma(\alpha)}.
\end{equation}
Given $k=(k_1,\ldots ,k_n)\in\Z_+^n$, let
\[
\delta(k)=(\underbrace{1,\ldots ,1}_{k_1},\ldots , \underbrace{n,\ldots ,n}_{k_n})\in W_{n,|k|}.
\]
Define a weight function $w_\bq\colon\Z_+^n\to\R_+$ by
\begin{equation}
\label{w_q}
w_\bq(k)=\min\{ |\lambda(\sigma,\delta(k))| : \sigma\in S_{|k|}\}.
\end{equation}
For example \cite[Proposition 5.12]{Pir_qfree}, if $|q_{ij}|\ge 1$ whenever $i<j$,
then $w_\bq(k)\equiv 1$, and if $|q_{ij}|\le 1$ whenever $i<j$, then
$w_\bq(k)=\prod_{i<j} q_{ij}^{k_i k_j}$. In particular, in the single-parameter case
(i.e., in the case where $q_{ij}=q$ for all $i<j$) we have
\begin{equation}
\label{w_q2}
w_q(k)=
\begin{cases}
1 & \text{if } |q|\ge 1,\\
|q|^{\sum_{i<j}k_i k_j} & \text{if } |q|<1.
\end{cases}
\end{equation}
As was shown in \cite[Theorem 5.11]{Pir_qfree}, for each multiplicatively antisymmetric complex
$n\times n$-matrix $\bq$ we have
\begin{equation}
\label{O_q}
\cO_\bq(\CC^n)=
\Bigl\{
a=\sum_{k\in\Z_+^n} c_k x^k :
\| a\|_t=\sum_{k\in\Z_+^n} |c_k| w_\bq(k) t^{|k|}<\infty
\;\forall t>0
\Bigr\}.
\end{equation}
The topology on $\cO_\bq(\CC^n)$ is given by the norms $\|\cdot\|_t \; (t>0)$.
Moreover, each norm $\|\cdot\|_t$ is submultiplicative.
\end{example}

\begin{example}
\label{ex:q_torus}
As in the previous example, let $\bq=(q_{ij})$ be a multiplicatively antisymmetric
complex $n\times n$-matrix.
Recall that the algebra $\cO_{\bq}^\reg((\CC^\times)^n)$ {\em of regular functions on the
quantum affine $n$-torus} is generated by
$2n$ elements $x_1^{\pm 1},\ldots ,x_n^{\pm 1}$ subject to the relations
$x_i x_j=q_{ij}x_j x_i$ for all $i,j$, and $x_i x_i^{-1}=x_i x_i^{-1}=1$ for all $i$
(see, e.g., \cite{Br_Good}).
If $q_{ij}=1$ for all $i,j$, then $\cO_\bq^\reg((\CC^\times)^n)$ is nothing but the
Laurent polynomial algebra $\CC[x_1^{\pm 1},\ldots ,x_n^{\pm 1}]=\cO^\reg((\CC^\times)^n)$.
In the general case, the Laurent monomials $x^k=x_1^{k_1}\cdots x_n^{k_n}\; (k\in\Z^n)$
form a basis of $\cO_\bq^\reg((\CC^\times)^n)$, and so
$\cO_\bq^\reg((\CC^\times)^n)$ may be viewed as
a ``deformed'' Laurent polynomial algebra.

Assume now that $|q_{ij}|=1$ for all $i,j$.
The Arens-Michael envelope of $\cO_\bq^\reg((\CC^\times)^n)$
is denoted by $\cO_\bq((\CC^\times)^n)$ and is called the
{\em algebra of holomorphic functions on the quantum affine $n$-torus}.
If $q_{ij}=1$ for all $i,j$, then $\cO_\bq((\CC^\times)^n)\cong\cO((\CC^\times)^n)$
(see Example~\ref{ex:AM_O}).
As was shown in \cite[Corollary 5.22]{Pir_qfree}, we have
\begin{equation}
\label{O_q_torus}
\cO_\bq((\CC^\times)^n)=
\Bigl\{
a=\sum_{k\in\Z^n} c_k x^k :
\| a\|_t=\sum_{k\in\Z^n} |c_k| t^{|k|}<\infty
\;\forall t>1
\Bigr\}.
\end{equation}
The topology on $\cO_\bq((\CC^\times)^n)$ is given by the norms $\|\cdot\|_t \; (t>1)$.
Moreover, each norm $\|\cdot\|_t$ is submultiplicative. Note that,
if $|q_{ij}|\ne 1$ for at least one pair of indices $i,j$, then the Arens-Michael envelope of
$\cO_\bq^\reg((\CC^\times)^n)$ is trivial \cite[Proposition 5.23]{Pir_qfree}.
\end{example}

We refer to \cite{Pir_qfree,Dos_hdim,Dos_absbas} for explicit descriptions of Arens-Michael envelopes
of some other finitely generated algebras, including quantum Weyl algebras,
the algebra of quantum $2\times 2$-matrices, and universal enveloping algebras.
Note that, in many concrete cases, the Arens-Michael envelope of a ``deformed polynomial
algebra'' can be interpreted as a ``deformed power series algebra'',
similarly to~\eqref{O_q} and~\eqref{O_q_torus}.

\subsection{Free polydisk}
\label{subsect:free_poly}
The following definition is motivated by~\eqref{F_O_free}.

\begin{definition}
\label{def:F_poly}
We define the {\em algebra of holomorphic functions on the free $n$-dimensional polydisk of
polyradius $R=(R_1,\ldots ,R_n)\in (0,+\infty]^n$} to be
\begin{equation}
\label{F_poly}
\cF(\DD^n_R)=\cO(\DD_{R_1})\Pfree\cdots\Pfree\cO(\DD_{R_n}).
\end{equation}
\end{definition}

By Corollary~\ref{cor:HFG_free}, $\cF(\DD^n_R)$ is an HFG algebra.
Letting $R_i=\infty$ for all $i$ and using~\eqref{F_O_free}, we see that $\cF(\CC^n)=\cF_n$.
Note that replacing in \eqref{F_O_free} and \eqref{F_poly} the Arens-Michael free product $\Pfree$
by the projective tensor product $\Ptens$ yields the algebras of holomorphic functions
on $\CC^n$ and $\DD_R^n$, respectively.

We have the canonical ``restriction'' map $\cF_n\to\cF(\DD^n_R)$ defined to be the
free product of the restriction maps $\cO(\CC)\to\cO(\DD_{R_i})\; (i=1,\ldots ,n)$.
For each $i=1,\ldots ,n$, the canonical image of the
free generator $\zeta_i\in\cF_n$ in $\cF(\DD^n_R)$ will also be denoted by $\zeta_i$.
The next result shows that $\cF(\DD^n_R)$ has a universal property similar to~\eqref{univ_F}.
In what follows, given an algebra $A$ and an element $a\in A$, the spectrum of $a$ in $A$
will be denoted by $\sigma_A(a)$.

\begin{prop}
\label{prop:univ_F_poly}
Let $A$ be an Arens-Michael algebra, and let $a=(a_1,\ldots ,a_n)$ be an $n$-tuple in $A^n$
such that $\sigma_A(a_i)\subseteq\DD_{R_i}$ for all $i=1,\ldots ,n$. Then there exists a unique
continuous homomorphism $\gamma_a^\free\colon\cF(\DD^n_R)\to A$
such that $\gamma_a^\free(\zeta_i)=a_i$ for all $i=1,\ldots ,n$. Moreover,
the assignment $a\mapsto\gamma_a^\free$ determines a natural isomorphism
\[
\Hom_\AM(\cF(\DD_R^n),A)\cong\{ a\in A^n : \sigma_A(a_i)\subseteq\DD_{R_i}\;\forall i=1,\ldots ,n\}
\qquad (A\in\AM).
\]
\end{prop}
\begin{proof}
For $n=1$, the above result is standard (see, e.g., \cite[VI.3, Theorem 3.2]{Mallios}).
Hence for each $n\in\N$ we have natural isomorphisms
\[
\begin{split}
\Hom_\AM(\cF(\DD_R^n),A))
&\cong\prod_{i=1}^n\Hom_\AM(\cO(\DD_{R_i}),A)\\
&\cong\{ a\in A^n : \sigma_A(a_i)\subseteq\DD_{R_i}\;\forall i=1,\ldots ,n\}. \qedhere
\end{split}
\]
\end{proof}

The algebra $\cF(\DD_R^n)$ can also be described more explicitly as follows.
Given $d\ge 2$ and $\alpha=(\alpha_1,\ldots ,\alpha_d)\in W_n$, let $s(\alpha)$ denote the
cardinality of the set
\[
\bigl\{ i \in \{ 1,\ldots ,d-1\} : \alpha_i\ne\alpha_{i+1} \bigr\}.
\]
If $|\alpha|\in\{ 0,1\}$, we set $s(\alpha)=|\alpha|-1$. Let also
$(0,R)=\prod_{i=1}^n (0,R_i)$.

\begin{prop}
We have
\[
\cF(\DD_R^n)=\Bigl\{ a=\sum_{\alpha\in W_n} c_\alpha\zeta_\alpha :
\| a\|_{\rho,\tau}=\sum_{\alpha\in W_n} |c_\alpha|\rho_\alpha \tau^{s(\alpha)+1}<\infty
\;\forall \rho\in (0,R),\; \forall \tau\ge 1\Bigr\}.
\]
The topology on $\cF(\DD_R^n)$ is given by the norms
$\|\cdot\|_{\rho,\tau}\; (\rho\in (0,R),\; \tau\ge 1)$,
and the multiplication is given by concatenation.
Moreover, each norm $\|\cdot\|_{\rho,\tau}$ is submultiplicative.
\end{prop}
\begin{proof}
Let $I=\{ 1,\ldots ,n\}$, and let $A_i=\cO(\DD_{R_i})$ for all $i\in I$.
Note that $A_i=\CC 1_{A_i}\oplus A_i^\circ$, where $A_i^\circ=\{ f\in A_i : f(0)=0\}$.
Hence Corollary~\ref{cor:Pfree_expl_fin} applies, and~\eqref{Pfree_expl_fin} holds.
We have
\[
A_i^\circ\cong\Bigl\{ f=\sum_{k=1}^\infty c_k \zeta_i^k :
\| f\|_{\rho}=\sum_{k=1}^\infty |c_k| \rho_i^k<\infty\; \forall \rho_i\in (0,R_i)\Bigr\}.
\]
Recall the standard fact that the projective tensor product of two K\"othe sequence
spaces is again a K\"othe sequence space (cf. \cite[41.7]{Kothe_II}).
Hence for each $\alpha=(\alpha_1,\ldots ,\alpha_d)\in I_\infty$ with $d>0$ we have
\[
A_\alpha\cong\Bigl\{ u=\sum_{k\in\N^d} c_k \zeta_{\alpha_1}^{k_1}\cdots \zeta_{\alpha_d}^{k_d} :
\| u\|_{\rho}^{(\alpha)}=\sum_{k\in\N^d}
|c_k|\rho_{\alpha_1}^{k_1}\cdots\rho_{\alpha_d}^{k_d}<\infty\; \forall \rho\in (0,R)\Bigr\}.
\]
Now \eqref{Pfree_expl_fin} yields
\begin{equation}
\label{Fball_expl_big}
\begin{split}
\cF(\DD_R^n)
\cong \Bigl\{ a&=\lambda 1+\sum_{d=1}^\infty \sum_{\alpha\in I_d} \sum_{k\in\N^d}
c_{\alpha k} \zeta_{\alpha_1}^{k_1}\cdots \zeta_{\alpha_d}^{k_d} :
\lambda\in\CC,\\
\| a\|_{\rho,\tau}&=|\lambda|+\sum_{d=1}^\infty \sum_{\alpha\in I_d} \sum_{k\in\N^d}
|c_{\alpha k}| \rho_{\alpha_1}^{k_1}\cdots\rho_{\alpha_d}^{k_d} \tau^d < \infty\;
\forall \rho\in (0,R),\; \forall \tau\ge 1\Bigr\}.
\end{split}
\end{equation}
Let $W_n^+=\bigsqcup_{d\ge 1} W_{n,d}$, and let
$S=\bigsqcup_{d\ge 1} (I_d\times\N^d)$.
We have a bijection between $S$ and $W_n^+$ given by
\begin{equation}
\label{bij}
\begin{split}
(\alpha,k)&=((\alpha_1,\ldots ,\alpha_d),(k_1,\ldots ,k_d))\in I_d\times\N^d\\
&\mapsto
\beta(\alpha)=\bigl(\underbrace{\alpha_1,\ldots ,\alpha_1}_{k_1},\ldots ,
\underbrace{\alpha_d,\ldots ,\alpha_d}_{k_d}\bigr)\in W_n^+.
\end{split}
\end{equation}
Note that $|\beta(\alpha)|=|k|$ and that $|\alpha|=s(\beta(\alpha))+1$.
Using~\eqref{bij} and~\eqref{Fball_expl_big}, we obtain
\begin{align*}
\cF(\DD_R^n)
&\cong\Bigl\{ a=\lambda 1+\sum_{\beta\in W_n^+} c_\beta\zeta_\beta : \lambda\in\CC,\\
& \qquad\qquad \| a\|_{\rho,\tau}=|\lambda|+
\sum_{\beta\in W_n^+} |c_\beta|\rho_\beta \tau^{s(\beta)+1}<\infty
\;\forall \rho\in (0,R),\; \forall \tau\ge 1\Bigr\}\\
&\cong \Bigl\{ a=\sum_{\beta\in W_n} c_\beta\zeta_\beta :
\| a\|_{\rho,\tau}=\sum_{\beta\in W_n} |c_\beta|\rho_\beta \tau^{s(\beta)+1}<\infty
\;\forall \rho\in (0,R),\; \forall \tau\ge 1\Bigr\}.
\end{align*}
The submultiplicativity of $\|\cdot\|_{\rho,\tau}$ follows from Corollary~\ref{cor:Pfree_expl_fin}.
\end{proof}

\begin{remark}
Another natural candidate for the algebra of holomorphic functions
on the free polydisk was introduced by J.~L.~Taylor \cite{T2,T3}.
By definition,
\begin{equation}
\label{F_poly_T}
\cF^T(\DD_R^n)=\Bigl\{ a=\sum_{\alpha\in W_n} c_\alpha\zeta_\alpha :
\| a\|_\rho=\sum_{\alpha\in W_n} |c_\alpha|\rho_\alpha<\infty
\;\forall \rho\in (0,R)\Bigr\}.
\end{equation}
Obviously, $\cF(\DD^n_R)\subseteq\cF^T(\DD^n_R)$, and
$\| a\|_\rho\le \| a\|_{\rho,\tau}$ for each $a\in\cF(\DD^n_R)$, each $\rho\in (0,R)$,
and each $\tau\ge 1$. Hence the embedding of $\cF(\DD^n_R)$ into $\cF^T(\DD^n_R)$
is continuous. For $n=1$, we clearly have $\cF^T(\DD_R^1)=\cF(\DD_R^1)=\cO(\DD_R^1)$.
To compare $\cF^T(\DD_R^n)$ with $\cF(\DD_R^n)$ for $n\ge 2$, it is convenient
to consider the following three cases.

(i) Suppose that $R_i=\infty$ for all $i$, i.e., $\DD_R^n=\CC^n$.
Comparing \eqref{Fn} and \eqref{F_poly_T}, we see that
$\cF^T(\CC^n)=\cF_n=\cF(\CC^n)$, both algebraically and topologically.

(ii) Suppose that $R_i<\infty$ for at most one $i\in\{ 1,\ldots ,n\}$.
We claim that $\cF^T(\DD_R^n)=\cF(\DD^n_R)$ in this case.
Without loss of generality,
assume that $R_1<\infty$ and $R_2=\ldots=R_n=\infty$. Given $\rho\in (0,R)$ and $\tau\ge 1$,
let $\rho'=(\rho_1,\tau^2\rho_2,\ldots ,\tau^2\rho_n)$. We clearly have $\rho'\in (0,R)$.
To prove the claim, it suffices to show that for each $\alpha\in W_n$ we have
\begin{equation}
\label{F=FT}
\|\zeta_\alpha\|_{\rho,\tau}\le\tau\|\zeta_\alpha\|_{\rho'}.
\end{equation}
Fix $\alpha\in W_n$ and write $\zeta_\alpha$ in the form
\[
\zeta_\alpha=x_1^{m_1}\zeta_{\alpha(1)}x_1^{m_2}\zeta_{\alpha(2)}
\cdots x_1^{m_k}\zeta_{\alpha(k)} x_1^{m_{k+1}},
\]
where $m_1,\ldots ,m_{k+1}\in\Z_+$, $\alpha(1),\ldots ,\alpha(k)\in W_n$
do not contain $1$, and $|\alpha(i)|\ge 1$ for all $i=1,\ldots ,k$.
We clearly have
\[
s(\alpha)\le 2k+\sum_{i=1}^k s(\alpha(i)).
\]
Since for each $\beta\in W_n$ we have $s(\beta)+1\le |\beta|$, it follows that
\[
s(\alpha)\le 2k+\sum_{i=1}^k s(\alpha(i))
\le k+\sum_{i=1}^k |\alpha(i)|\le 2\sum_{i=1}^k |\alpha(i)|.
\]
Therefore
\begin{equation}
\label{F=FT_est1}
\|\zeta_\alpha\|_{\rho,\tau}
=\rho_\alpha \tau^{s(\alpha)+1}
\le \tau \rho_\alpha \tau^{2\sum |\alpha(i)|},
\end{equation}
while
\begin{equation}
\label{F=FT_est2}
\|\zeta_\alpha\|_{\rho'}
=\rho_1^{\sum m_j} \tau^{2|\alpha(1)|}\rho_{\alpha(1)}\cdots\tau^{2|\alpha(k)|}\rho_{\alpha(k)}
=\rho_\alpha \tau^{2\sum |\alpha(i)|}.
\end{equation}
Comparing \eqref{F=FT_est1} with \eqref{F=FT_est2}, we obtain \eqref{F=FT},
as required. Therefore $\cF^T(\DD_R^n)=\cF(\DD^n_R)$, both algebraically and topologically.

(iii) Finally, suppose that $R_i<\infty$ for at least two $i\in\{ 1,\ldots ,n\}$. As was shown
in~\cite{T2,Lum}, $\cF^T(\DD_R^n)$ is not nuclear in this case.
Hence $\cF^T(\DD_R^n)$ is not an HFG algebra, and so $\cF^T(\DD_R^n)\ne\cF(\DD^n_R)$.
Moreover, the topology on $\cF(\DD^n_R)$
is strictly stronger than that inherited from $\cF^T(\DD_R^n)$.
\end{remark}

\subsection{$q$-products of balanced and Reinhardt domains}
\label{subsect:q-prod}
In the remaining subsections, we will introduce deformed products on some algebras
of holomorphic functions, and we will show that the resulting deformed algebras are HFG algebras.
Let us start with the simplest case of the algebra $\cO(D_1\times D_2)$,
where $D_1\subseteq\CC^m$ and $D_2\subseteq\CC^n$ are balanced domains.

\begin{prop}
\label{prop:q-prod_bal}
Let $D_1\subseteq\CC^m$ and $D_2\subseteq\CC^n$ be balanced domains.
Denote by $z_1,\ldots ,z_m$ (respectively, $w_1,\ldots ,w_n$) the coordinates on $\CC^m$
(respectively, $\CC^n$). Then for each $q\in\CC^\times$ there exists a unique continuous
multiplication on $\cO(D_1\times D_2)$ such that
\begin{align}
z_i z_j&=z_j z_i \qquad (i,j=1,\ldots ,m),\notag\\
\label{q-bal}
w_i w_j&=w_j w_i \qquad (i,j=1,\ldots ,n),\\
z_i w_j&=qw_j z_i \qquad (i=1,\ldots ,m,\; j=1,\ldots ,n).\notag
\end{align}
The resulting Fr\'echet algebra $\cO(D_1\times_q D_2)$ is an HFG algebra.
\end{prop}
\begin{proof}
By interchanging $D_1$ and $D_2$ if necessary, we can assume that $|q|\ge 1$.
Consider the endomorphism $\sigma$ of $\cO(D_1)$ given by $(\sigma f)(z)=f(q^{-1}z)$.
For each compact set $K\subset D_1$ and each $f\in\cO(D_1)$ with homogeneous
expansion $f=\sum_{m\in\Z_+} f_m$, we have
\[
\| \sigma(f) \|_K'
=\sum_{m\in\Z_+} \| \sigma(f_m) \|_K
=\sum_{m\in\Z_+} |q|^{-m} \| f_m \|_K
\le\sum_{m\in\Z_+} \| f_m \|_K = \| f\|_K'.
\]
Hence~\eqref{suff_equi_semi} holds, and the subsemigroup of $\End(\cO(D_1))$ generated by
$\sigma$ is equicontinuous. Now~Proposition~\ref{prop:skew_hol_bal} yields the HFG algebra
$\cO(D_2,\cO(D_1);\sigma)$. Identifying $\cO(D_2,\cO(D_1);\sigma)$ with
$\cO(D_1\times D_2)$ as a Fr\'echet space (see Section~\ref{sect:prelim}),
we see that the relations~\eqref{q-bal} hold
in $\cO(D_2,\cO(D_1);\sigma)$. Letting $\cO(D_1\times_q D_2)=\cO(D_2,\cO(D_1);\sigma)$,
we obtain the required algebra.
\end{proof}

If we assume that $D_1$ and $D_2$ are Reinhardt domains,
then we can generalize the commutation relations~\eqref{q-bal} as follows.

\begin{prop}
Let $D_1\subseteq\CC^m$ and $D_2\subseteq\CC^n$ be Reinhardt domains.
Denote by $z_1,\ldots ,z_m$ (respectively, $w_1,\ldots ,w_n$) the coordinates on $\CC^m$
(respectively, $\CC^n$). Let $\bq=(q_{ij})$ be a complex $m\times n$-matrix, and assume that one of
the following conditions holds:
\begin{mycompactenum}
\item $D_1$ and $D_2$ are complete, and either $|q_{ij}|\ge 1$ for all $i,j$,
or $|q_{ij}|\le 1$ for all $i,j$;
\item $|q_{ij}|=1$ for all $i,j$.
\end{mycompactenum}
Then there exists a unique continuous
multiplication on $\cO(D_1\times D_2)$ such that
\begin{align}
z_i z_j&=z_j z_i \qquad (i,j=1,\ldots ,m),\notag\\
\label{q-Rnhrdt}
w_i w_j&=w_j w_i \qquad (i,j=1,\ldots ,n),\\
z_i w_j&=q_{ij}w_j z_i \qquad (i=1,\ldots ,m,\; j=1,\ldots ,n).\notag
\end{align}
The resulting Fr\'echet algebra $\cO(D_1\times_\bq D_2)$ is an HFG algebra.
\end{prop}
\begin{proof}
(i) By interchanging $D_1$ and $D_2$ if necessary, we can assume that $|q_{ij}|\ge 1$ for all $i,j$.
Consider the commuting endomorphisms $\sigma_1,\ldots ,\sigma_n$ of $\cO(D_1)$ given by
$(\sigma_j f)(z)=f(q_{1j}^{-1}z_1,\ldots ,q_{mj}^{-1}z_m)$.
For each compact set $K\subset D_1$ and each $f\in\cO(D_1)$ with Taylor
expansion $f=\sum_{k\in\Z_+^m} c_k z^k$, we have
\[
\| \sigma_j(f) \|_K''
=\sum_{k\in\Z_+^m} |c_k| \| \sigma_j(z^k) \|_K
\le\sum_{k\in\Z_+^m} |c_k| \| z^k\|_K = \| f\|_K''.
\]
Hence~\eqref{suff_equi_semi} holds, and the subsemigroup of $\End(\cO(D_1))$ generated by
$\sigma_1,\ldots,\sigma_n$ is equicontinuous.
Now Proposition~\ref{prop:skew_hol_Rnhrdt} (i) yields the HFG algebra
$\cO(D_2,\cO(D_1);\sigma)$. Identifying $\cO(D_2,\cO(D_1);\sigma)$ with
$\cO(D_1\times D_2)$ as a Fr\'echet space (see Section~\ref{sect:prelim}),
we see that the relations~\eqref{q-Rnhrdt} hold
in $\cO(D_2,\cO(D_1);\sigma)$. Letting $\cO(D_1\times_\bq D_2)=\cO(D_2,\cO(D_1);\sigma)$,
we obtain the required algebra.

(ii) The proof is similar to (i), the only difference is that now $\sigma_1,\ldots ,\sigma_n$
are automorphisms satisfying~\eqref{suff_equi_grp}, and
so the subgroup of $\Aut(\cO(D_1))$ generated by
$\sigma_1,\ldots,\sigma_n$ is equicontinuous.
\end{proof}

\subsection{Quantum polydisk}
\label{subsect:q_polydisk}

Our next example is motivated by~\eqref{O_q}.
Let $\bq=(q_{ij})$ be a multiplicatively antisymmetric complex $n\times n$-matrix.
Observe that, for each $t>0$,
the norm $\|\cdot\|_t$ on $\cO_\bq^\reg(\CC^n)$ given by~\eqref{O_q}
is a special case of the following one. Given $\rho\in (0,+\infty)^n$ and
$a=\sum_k c_k x^k\in\cO_\bq^\reg(\CC^n)$, let
\[
\| a\|_\rho=\sum_{k\in\Z_+^n} |c_k| w_\bq(k) \rho^{k},
\]
where the function $w_\bq\colon \Z_+^n\to\R_+$ is given by \eqref{w_q}.
The same argument as in \cite[Lemma 5.10]{Pir_qfree} shows that $\|\cdot\|_\rho$
is submultiplicative.

\begin{definition}
Let $R\in (0,+\infty]^n$.
We define the {\em algebra of holomorphic functions on the
quantum $n$-polydisk of polyradius $R$} by
\[
\cO_\bq(\DD^n_R)=
\Bigl\{
a=\sum_{k\in\Z_+^n} c_k x^k :
\| a\|_\rho=\sum_{k\in\Z_+^n} |c_k| w_\bq(k) \rho^{k}<\infty
\;\forall \rho\in (0,R) \Bigr\}.
\]
The topology on $\cO_\bq(\DD^n_R)$ is given by the norms $\|\cdot\|_\rho\; (\rho\in (0,R))$,
and the multiplication on $\cO_\bq(\DD^n_R)$ is uniquely determined by $x_i x_j=q_{ij}x_j x_i$ for all $i,j$.
\end{definition}

In other words, $\cO_\bq(\DD^n_R)$ is the completion of $\cO_\bq^\reg(\CC^n)$
with respect to the family $\{\|\cdot\|_\rho :\rho\in (0,R)\}$ of submultiplicative norms.
Clearly, if $q_{ij}=1$ for all $i,j$, then $\cO_\bq(\DD^n_R)$ is topologically isomorphic to the algebra
$\cO(\DD^n_R)$ of holomorphic functions on the polydisk $\DD_R^n$.

\begin{theorem}
\label{thm:q_poly_quot_free_poly}
There exists a unique continuous homomorphism
\[
\pi\colon\cF(\DD_R^n)\to\cO_\bq(\DD_R^n)\quad\text{such that}
\quad\pi(\zeta_i)=x_i \quad (i=1,\ldots ,n).
\]
Moreover, $\pi$ is onto, and $\Ker\pi$ coincides with
the closed two-sided ideal of $\cF(\DD_R^n)$ generated by the elements
$\zeta_i\zeta_j-q_{ij}\zeta_j\zeta_i$ for all $i,j=1,\ldots ,n$.
Finally, $\Ker\pi$ is a complemented subspace of $\cF(\DD_R^n)$.
\end{theorem}

To prove Theorem~\ref{thm:q_poly_quot_free_poly}, we need some preparation.
Given $k\in\Z_+^n$, let
\[
W(k)=\{ \sigma(\delta(k)) : \sigma\in S_{|k|},\; |\lambda(\sigma,\delta(k))|=w_\bq(k)\}
\]
(for notation, see Example~\ref{ex:q_aff}).
Observe that, if $\alpha\in W_n$ and $\sigma_1,\sigma_2\in S_{|\alpha|}$ are such that
$\sigma_1(\alpha)=\sigma_2(\alpha)$, then $\lambda(\sigma_1,\alpha)=\lambda(\sigma_2,\alpha)$
(see \eqref{lambda}). In other words, $\lambda(\sigma,\alpha)$ depends only on $\alpha$
and $\sigma(\alpha)$. As a consequence we obtain the following.

\begin{lemma}
\label{lemma:Wk}
If $\sigma\in S_{|k|}$, then $\sigma(\delta(k))\in W(k)$ if and only if
$|\lambda(\sigma,\delta(k))|=w_\bq(k)$.
\end{lemma}

Recall \cite[Lemma 5.8]{Pir_qfree}
that, for each $\alpha\in W_n$ and each $\sigma,\tau\in S_{|\alpha|}$ we have
\begin{equation}
\label{cocycle}
\lambda(\sigma\tau,\alpha)=\lambda(\sigma,\tau(\alpha))\lambda(\tau,\alpha).
\end{equation}

\begin{lemma}
\label{lemma:lambda=1}
If $\alpha\in W(k)$, then
\begin{mycompactenum}
\item $|\lambda(\sigma,\alpha)|\ge 1$ for each $\sigma\in S_{|k|}$;
\item $|\lambda(\sigma,\alpha)|=1$ if and only if $\sigma(\alpha)\in W(k)$.
\end{mycompactenum}
\end{lemma}
\begin{proof}
Choose $\tau\in S_{|k|}$ such that $\alpha=\tau(\delta(k))$.
By Lemma~\ref{lemma:Wk}, we have $|\lambda(\tau,\delta(k))|=w_\bq(k)$.
Moreover, $\sigma(\alpha)\in W(k)$ if and only if $|\lambda(\sigma\tau,\delta(k))|=w_\bq(k)$.
Using~\eqref{cocycle}, we see that
\[
w_\bq(k)\le |\lambda(\sigma\tau,\delta(k))|=|\lambda(\sigma,\alpha)| w_\bq(k).
\]
This readily implies both (i) and (ii).
\end{proof}

Let $d\in \Z_+$, and let $d_1,\ldots ,d_p\in \Z_+$ be such that $d_1+\cdots +d_p=d$.
We will identify each $p$-tuple $(\beta_1,\ldots ,\beta_p)$ such that
$\beta_i\in W_{n,d_i}\; (i=1,\ldots ,p)$ with the element of $W_{n,d}$ obtained from
$\beta_1,\ldots ,\beta_p$ by concatenation.

\begin{definition}
Let $\ell,d\in \Z_+$, $\ell\le d$. We say that $\beta\in W_{n,\ell}$ is a {\em subword}
of $\alpha\in W_{n,d}$ if there exist $\ell',\ell''\in \Z_+$ and $\beta'\in W_{n,\ell'}$,
$\beta''\in W_{n,\ell''}$ such that $d=\ell'+\ell+\ell''$ and $\alpha=(\beta',\beta,\beta'')$.
\end{definition}

\begin{definition}
Given $j\in\{ 1,\ldots ,n\}$, a {\em $j$-word} is a word of the form
$(j,\ldots ,j)\in W_{n,d}$ for some $d\in\Z_+$.
\end{definition}

\begin{definition}
Let $\alpha\in W_{n,d}$, and let $\beta$ be a $j$-subword of $\alpha$.
We say that $\beta$ is a {\em maximal $j$-subword} of $\alpha$ if $\beta$ is not a proper
subword any other $j$-subword of $\alpha$.
\end{definition}

Given $\alpha\in W_{n,d}$ and $j\in\{ 1,\ldots ,n\}$, let $N(\alpha,j)$ denote the
number of maximal nonempty $j$-subwords of $\alpha$. Let also
\begin{align*}
\bnc(\alpha)&=\bigl\{ j\in\{ 1,\ldots ,n\} : N(\alpha,j)\ge 2\bigr\},\\
\bc(\alpha)&=\{ 1,\ldots ,n\}\setminus\bnc(\alpha).
\end{align*}

\begin{definition}
We say that $\alpha\in W_{n,d}$ is {\em compact} if $\bnc(\alpha)=\varnothing$.
Equivalently, $\alpha$ is compact if there exist $\sigma\in S_n$ and
$k=(k_1,\ldots ,k_n)\in\Z_+^n$ such that
\[
\alpha=(\underbrace{\sigma(1),\ldots ,\sigma(1)}_{k_1},\ldots ,
\underbrace{\sigma(n),\ldots ,\sigma(n)}_{k_n}).
\]
\end{definition}

Given $s,t,d\in\Z_+$ such that $1\le s< t\le d$, consider the cycle
\[
\sigma_{s,t}=(s\; s+1 \ldots t-1\; t)\in S_d.
\]
Let also $\sigma_{t,s}=\sigma_{s,t}^{-1}$.
We clearly have
\begin{align*}
\sigma_{s,t}&=(s\; s+1)(s+1\; s+2)\cdots (t-1\; t),\\
\sigma_{t,s}&=(t-1\; t)(t-2\; t-1)\cdots (s\; s+1).
\end{align*}

\begin{lemma}
For each $\alpha\in W_{n,d}$ we have
\begin{equation}
\label{lambda_sigma_st}
\lambda(\sigma_{s,t},\alpha)=\prod_{i=s}^{t-1} q_{\alpha_i \alpha_t},\quad
\lambda(\sigma_{t,s},\alpha)=\prod_{i=s+1}^t q_{\alpha_s \alpha_i}.
\end{equation}
\end{lemma}
\begin{proof}
We prove only the first equality, the proof of the second being similar.
We use induction on $p=t-s$. For $p=1$, there is nothing to prove.
Let now $p\ge 2$, and let $\alpha'=(t-1\; t)\alpha$. Since $\sigma_{s,t}=\sigma_{s,t-1} (t-1\; t)$,
the induction hypothesis and \eqref{cocycle} yield
\[
\begin{split}
\lambda(\sigma_{s,t},\alpha)
&=\lambda(\sigma_{s,t-1},\alpha')\lambda((t-1\; t),\alpha)\\
&=\Bigl(\prod_{i=s}^{t-2} q_{\alpha'_i \alpha'_{t-1}}\Bigr) q_{\alpha_{t-1}\alpha_t}
=\Bigl(\prod_{i=s}^{t-2} q_{\alpha_i \alpha_t}\Bigr) q_{\alpha_{t-1}\alpha_t}
=\prod_{i=s}^{t-1} q_{\alpha_i \alpha_t}. \qedhere
\end{split}
\]
\end{proof}

\begin{lemma}
\label{lemma:compactify}
Let $k\in \Z_+^n$. For each noncompact $\alpha\in W(k)$ and each $j\in\bnc(\alpha)$
there exists $\sigma\in S_{|k|}$ such that $\sigma(\alpha)\in W(k)$,
$N(\sigma(\alpha),j)=N(\alpha,j)-1$, and $\bc(\alpha)\subseteq\bc(\sigma(\alpha))$.
\end{lemma}
\begin{proof}
Since $j\in\bnc(\alpha)$, we can write $\alpha$ in the form
\begin{equation}
\label{alpha_begin}
\alpha=(\beta_1,\gamma_1,\beta_2,\gamma_2,\beta_3)
\end{equation}
where $\gamma_1$ and $\gamma_2$ are maximal nonempty $j$-subwords of $\alpha$,
and $\beta_1,\beta_2,\beta_3$ are subwords of $\alpha$ such that $\beta_2\ne\varnothing$.
Let also
\[
r=|\beta_1|,\quad s=|\beta_1|+|\gamma_1|,\quad t=|\beta_1|+|\gamma_1|+|\beta_2|.
\]
In particular, we have $\alpha_s=\alpha_{t+1}=j$.
Using~\eqref{lambda_sigma_st}, we obtain
\[
\lambda(\sigma_{t,s},\alpha)=\prod_{i=s+1}^t q_{j\alpha_i}, \quad
\lambda(\sigma_{s+1,t+1},\alpha)=\prod_{i=s+1}^t q_{\alpha_i j}.
\]
Hence $\lambda(\sigma_{s+1,t+1},\alpha)=\lambda(\sigma_{t,s},\alpha)^{-1}$.
Applying Lemma~\ref{lemma:lambda=1}, we conclude that
$|\lambda(\sigma_{t,s},\alpha)|=|\lambda(\sigma_{s+1,t+1},\alpha)|=1$,
whence $\sigma_{t,s}(\alpha)\in W(k)$.

Let $\alpha'=\sigma_{t,s}(\alpha)$. We have
\[
\alpha'=(\beta_1,\gamma_1',\beta_2,\gamma_2',\beta_3),
\]
where $\gamma_2'=(j,\gamma_2)$ and $\gamma_1=(\gamma_1',j)$.
If $\gamma_1'\ne\varnothing$, then the same procedure applied to
$\alpha'$ and to the maximal nonempty $j$-subwords $\gamma_1',\gamma_2'$
yields
\[
\alpha''=\sigma_{t-1,s-1}(\alpha')=(\beta_1,\gamma_1'',\beta_2,\gamma_2'',\beta_3)\in W(k),
\]
where $\gamma_2''=(j,\gamma_2')$ and $\gamma_1'=(\gamma_1'',j)$.
After finitely many steps we obtain
\begin{equation}
\label{alpha_end}
\alpha^{(s-r)}=\sigma_{t-s+r+1,r+1}(\alpha^{(s-r+1)})
=(\beta_1,\beta_2,\gamma_1,\gamma_2,\beta_3)\in W(k).
\end{equation}
We clearly have $\alpha^{(s-r)}=\sigma(\alpha)$, where
$\sigma=\prod_{i=1}^{s-r}\sigma_{t-s+r+i,r+i}$. Comparing~\eqref{alpha_end}
with~\eqref{alpha_begin}, we conclude that $N(\alpha^{(s-r)},j)=N(\alpha,j)-1$
and $\bc(\alpha)\subseteq\bc(\alpha^{(s-r)})$.
\end{proof}

\begin{lemma}
\label{lemma:compact_exists}
For each $k\in\Z_+^n$, $W(k)$ contains a compact word.
\end{lemma}
\begin{proof}
Choose $\alpha\in W(k)$ such that $|\bc(\alpha)|\ge |\bc(\beta)|$ for each $\beta\in W(k)$.
Assume, towards a contradiction, that $\alpha$ is noncompact, and
fix any $j\in\bnc(\alpha)$.
Applying Lemma~\ref{lemma:compactify} finitely many times,
we obtain $\beta\in W(k)$ such that $N(\beta,j)=1$ and $\bc(\alpha)\subsetneq\bc(\beta)$.
Thus $|\bc(\beta)|>|\bc(\alpha)|$. The resulting contradiction completes the proof.
\end{proof}

\begin{proof}[Proof of Theorem {\upshape\ref{thm:q_poly_quot_free_poly}}]
Fix $i\in\{1,\ldots ,n\}$, and let $z_i\in\cO(\DD_{R_i})$ denote the complex coordinate.
Let $e_i\in\Z_+^n$ denote the $n$-tuple with $1$ at the $i$th
position, $0$ elsewhere. Clearly, for each $d\in\Z_+$ the element
$\delta(de_i)=(i,\ldots ,i)\in W_{n,d}$ is invariant under the action of $S_d$.
Hence $w_\bq(de_i)=1$, and
\[
\| x_i^d\|_\rho=\| x^{de_i}\|_\rho=\rho_i^d=\| z_i^d\|_{\rho_i} \qquad (\rho\in (0,R)).
\]
This implies that there exists a Fr\'echet algebra embedding
$\cO(\DD_{R_i})\hookrightarrow\cO_\bq(\DD_R^n)$
that sends the complex coordinate $z_i$ to $x_i$. Hence, by Definition~\ref{def:F_poly},
there exists a unique continuous homomorphism
\[
\pi\colon\cF(\DD_R^n)\to\cO_\bq(\DD_R^n),\quad \zeta_i\mapsto x_i \quad (i=1,\ldots ,n).
\]
Let us construct a continuous linear map
\begin{equation}
\label{sec}
\varkappa\colon\cO_\bq(\DD_R^n)\to\cF(\DD_R^n) \quad\text{such that}\quad \pi \varkappa=\id.
\end{equation}
To this end, fix any $k\in\Z_+^n$ and
choose a compact word $\alpha_k\in W(k)$ (see Lemma~\ref{lemma:compact_exists}).
Let $\sigma_k\in S_{|k|}$ be such that $\alpha_k=\sigma_k(\delta(k))$. Then
\begin{equation}
\label{pi(zeta)}
x^k=x_{\delta(k)}=\lambda(\sigma_k,\delta(k))x_{\alpha_k}
=\pi(\lambda(\sigma_k,\delta(k))\zeta_{\alpha_k}).
\end{equation}
We claim that for each $a=\sum_{k\in\Z_+^n} c_k x^k\in\cO_\bq(\DD_R^n)$
the series
\begin{equation}
\label{ser1}
\sum_{k\in\Z_+^n} c_k \lambda(\sigma_k,\delta(k))\zeta_{\alpha_k}
\end{equation}
absolutely converges in $\cF(\DD^n_R)$. Indeed, since $\alpha_k$ is compact,
it follows that $s(\alpha_k)\le n-1$ for each $k\in\Z_+^n$.
By Lemma~\ref{lemma:Wk}, we also have $|\lambda(\sigma_k,\delta(k))|=w_\bq(k)$.
Hence for each $\rho\in (0,R)$ and each $\tau\ge 1$ we obtain
\begin{equation}
\label{absconv1}
\sum_{k\in\Z_+^n} \| c_k \lambda(\sigma_k,\delta(k))\zeta_{\alpha_k}\|_{\rho,\tau}
\le\sum_{k\in\Z_+^n} |c_k| w_\bq(k) \rho^{k} \tau^n =\tau^n \| a\|_\rho.
\end{equation}
Therefore~\eqref{ser1} converges to an element $\varkappa(a)\in\cF(\DD_R^n)$.
It is also immediate from~\eqref{absconv1} that for each $a\in\cO_\bq(\DD_R^n)$
we have $\| \varkappa(a)\|_{\rho,\tau}\le \tau^n \| a\|_\rho$,
whence $\varkappa\colon\cO_\bq(\DD_R^n)\to\cF(\DD_R^n)$ is a continuous linear map.
Using~\eqref{pi(zeta)}, we see that $\pi \varkappa=\id$.
Hence $\pi$ is onto, and $\Ker\pi$ is a complemented subspace of $\cF(\DD^n_R)$.

Let now $I\subset\cF(\DD_R^n)$ denote the closed two-sided ideal generated by the elements
$\zeta_i\zeta_j-q_{ij}\zeta_j\zeta_i$ for all $i,j$. Clearly, $I\subseteq\Ker\pi$, and hence
$\pi$ induces a continuous homomorphism
\[
\bar\pi\colon\cF(\DD_R^n)/I\to\cO_\bq(\DD_R^n),\quad \bar\zeta_i\mapsto x_i \quad (i=1,\ldots ,n),
\]
where $\bar\zeta_i=\zeta_i+I\in\cF(\DD_R^n)/I$.
Let $\bar\varkappa\colon \cO_\bq(\DD_R^n)\to\cF(\DD_R^n)/I$ denote the composition of $\varkappa$
with the quotient map $\cF(\DD_R^n)\to\cF(\DD_R^n)/I$. It is immediate from~\eqref{sec}
that $\bar\pi\bar\varkappa=\id$. On the other hand, it follows from the definition of $I$
that the linear span of the set $\{ \bar\zeta_{\alpha_k} : k\in\Z_+^n\}$
is dense in $\cF(\DD_R^n)/I$. Hence $\Im\bar\varkappa$ is also dense in $\cF(\DD_R^n)/I$.
Together with $\bar\pi\bar\varkappa=\id$, this implies that $\bar\pi$ and $\bar\varkappa$
are topological isomorphisms. Therefore $I=\Ker\pi$. This completes the proof.
\end{proof}

\begin{corollary}
\label{cor:qpoly_HFG}
$\cO_\bq(\DD^n_R)$ is an HFG algebra.
\end{corollary}
\begin{proof}
Immediate from Theorem \ref{thm:q_poly_quot_free_poly}, Corollary~\ref{cor:HFG_free},
and Proposition~\ref{prop:HFG_quot}.
\end{proof}

\subsection{Quantum polyannulus}
\label{subsect:q_polyann}

Our last example is motivated by~\eqref{O_q_torus}.
Let $\bq=(q_{ij})$ be a multiplicatively antisymmetric complex $n\times n$-matrix
such that $|q_{ij}|=1$ for all $i,j$.
For each $k\in\Z^n$, define $k^+,k^-\in\Z^n$ by
$k^+_i=\max\{ k_i,0\}$ and $k^-_i=\min\{ k_i,0\}$.
Given $r,R\in [0,+\infty]^n$, we write $r<R$ if $r_i<R_i$ for all $i$. In this case, let
$(r,R)=\prod_{i=1}^n (r_i,R_i)$. Given $\rho,\tau\in (0,+\infty)^n$ with $\rho<\tau$
and $a=\sum_k c_k x^k\in\cO_\bq^\reg((\CC^\times)^n)$, let
\[
\| a\|_{\rho,\tau}=\sum_{k\in\Z^n} |c_k| \rho^{k^-}\tau^{k^+}.
\]
Clearly, $\|\cdot\|_{\rho,\tau}$ is a norm on $\cO_\bq^\reg((\CC^\times)^n)$.
Observe that for each $k\in\Z^n$ we have
\[
\| x^k\|_{\rho,\tau}=\rho^{k^-}\tau^{k^+}=\max\{ |z^k| : z\in\bar\DD^n_{\rho,\tau}\}.
\]
Since $|q_{ij}|=1$ for all $i,j$, it follows that
\[
\| x^k x^\ell\|_{\rho,\tau} =\| x^{k+\ell}\|_{\rho,\tau}
\le \| x^k\|_{\rho,\tau} \| x^\ell\|_{\rho,\tau},
\]
whence the norm $\|\cdot\|_{\rho,\tau}$ is submultiplicative on $\cO_\bq^\reg((\CC^\times)^n)$.
Observe that, for each $t>1$, the norm $\|\cdot\|_t$ given by~\eqref{O_q_torus}
is a special case of the above construction, since
\[
\| \cdot\|_t=\| \cdot\|_{(t^{-1},\ldots ,t^{-1}),(t,\ldots ,t)}.
\]

\begin{definition}
Let $r,R\in [0,+\infty]^n,\; r<R$.
We define the {\em algebra of holomorphic functions on the
quantum $n$-polyannulus of polyradii $r$ and $R$} by
\[
\cO_\bq(\DD^n_{r,R})=
\Bigl\{
a=\sum_{k\in\Z^n} c_k x^k :
\| a\|_{\rho,\tau}=\sum_{k\in\Z^n} |c_k| \rho^{k^-}\tau^{k^+}<\infty
\;\forall \rho,\tau\in (r,R),\; \rho<\tau \Bigr\}.
\]
The topology on $\cO_\bq(\DD^n_{r,R})$ is given by
the norms $\|\cdot\|_{\rho,\tau}\; (\rho,\tau\in (r,R),\; \rho<\tau)$,
and the multiplication on $\cO_\bq(\DD^n_{r,R})$ is uniquely determined
by $x_i x_j=q_{ij}x_j x_i$ for all $i,j$.
\end{definition}

In other words, $\cO_\bq(\DD^n_{r,R})$ is the completion of $\cO_\bq^\reg((\CC^\times)^n)$
with respect
to the family $\{\|\cdot\|_{\rho,\tau} :\rho,\tau\in (r,R),\; \rho<\tau\}$ of submultiplicative norms.
Observe that, as a Fr\'echet space, $\cO_\bq(\DD^n_{r,R})$ is
topologically isomorphic to the space
$\cO(\DD^n_{r,R})$ of holomorphic functions on the polyannulus $\DD_{r,R}^n$.
Explicitly, the isomorphism takes each holomorphic function on $\DD_{r,R}^n$
to its Laurent expansion about $0$.
Clearly, if $q_{ij}=1$ for all $i,j$, then the above isomorphism is a Fr\'echet algebra
isomorphism.

\begin{theorem}
\label{thm:qann_HFG}
$\cO_\bq(\DD^n_{r,R})$ is an HFG algebra.
\end{theorem}
\begin{proof}
For $n=1$, there is nothing to prove. Let $n\ge 2$, and assume that the result holds for $n-1$.
Let $r'=(r_1,\ldots ,r_{n-1})$, $R'=(R_1,\ldots ,R_{n-1})$, and $\bq'=(q_{ij})_{i,j\le n-1}$.
Let $\sigma$ be the automorphism of $\cO_{\bq'}^\reg((\CC^\times)^{n-1})$
given by $\sigma(x_i)=q_{in}^{-1} x_i\; (i=1,\ldots ,n-1)$.
Clearly, for each $a\in \cO_{\bq'}^\reg((\CC^\times)^{n-1})$ we have
\begin{equation}
\label{sigma_isometr}
\| \sigma(a)\|_{\rho,\tau}=\| a\|_{\rho,\tau} \qquad (\rho,\tau\in (r',R'),\; \rho<\tau).
\end{equation}
Hence $\sigma$ uniquely extends to a topological automorphism of $\cO_{\bq'}(\DD^{n-1}_{r',R'})$.
Moreover, \eqref{sigma_isometr} holds for each $a\in\cO_{\bq'}(\DD^{n-1}_{r',R'})$.
Hence $\sigma$ satisfies~\eqref{suff_equi_grp}, and the subgroup of
$\Aut(\cO_{\bq'}(\DD^{n-1}_{r',R'}))$ generated by $\sigma$ is equicontinuous.
Applying Proposition~\ref{prop:skew_hol_Rnhrdt} (ii), we obtain the HFG algebra
$\cO(\DD_{r_n,R_n},\cO_{\bq'}(\DD^{n-1}_{r',R'});\sigma)$, which is easily seen
to be topologically isomorphic to $\cO_\bq(\DD^n_{r,R})$.
\end{proof}

\begin{remark}
If $|q_{ij}|\ne 1$ for at least one pair of indices $i,j$, then the only submultiplicative seminorm
on $\cO_\bq^\reg((\CC^\times)^n)$ is identically zero (see Example~\ref{ex:q_torus}).
That is why we make no attempt to define $\cO_\bq(\DD^n_{r,R})$
in this case.
\end{remark}

\begin{remark}
If $\bq=(q_{ij})$ is a multiplicatively antisymmetric complex $n\times n$-matrix satisfying
$|q_{ij}|\ge 1$ for all $i<j$, then Corollary~\ref{cor:qpoly_HFG} can be proved similarly
to Theorem~\ref{thm:qann_HFG}. In the general case, however,
the endomorphism $\sigma$ of $\cO_{\bq'}^\reg(\CC^{n-1})$ given
by $\sigma(x_i)=q_{in}^{-1} x_i\; (i=1,\ldots ,n-1)$ need not extend to a continuous
endomorphism of $\cO_{\bq'}(\DD^{n-1}_{R'})$.
Consider, for example, the simplest case where $n=2$, $R'=1$, and $|q|=|q_{12}|<1$.
We have $\cO_{\bq'}^\reg(\CC^{n-1})=\CC[x_1]$,
$\cO_{\bq'}(\DD^{n-1}_{R'})=\cO(\DD_1)$, and $\sigma\colon\CC[x_1]\to\CC[x_1]$
acts by $\sigma(x_1)=q^{-1}x_1$. If $\sigma$ were continuous with respect to the topology
inherited from $\cO(\DD_1)$, then for each $\rho\in (0,1)$ there would exist $r\in (\rho,1)$ and $C>0$
such that $\| \sigma(a)\|_\rho\le C\| a\|_r$ for each $a\in\CC[x_1]$.
For $a=x_1^k$, this yields $|q|^{-k} \rho^k \le Cr^k$, or, equivalently,
$|q|^{-1}\le C^{1/k} (r/\rho)$. Letting first $k\to\infty$ and then
$\rho\to 1$, we obtain $|q|^{-1}\le 1$.
The resulting contradiction shows that $\sigma$ is not continuous. Thus the above method
cannot be applied to Corollary~\ref{cor:qpoly_HFG} in the general case.
\end{remark}

\end{document}